\documentclass[11pt]{article}

\usepackage{amsmath,amssymb,amsfonts}
\usepackage{color,xcolor}
\usepackage{mathrsfs}
\usepackage{algorithm,algorithmic,float}
\usepackage{graphicx,subfigure,epsfig,psfrag,fancyhdr,enumerate}
\usepackage[linktocpage,colorlinks,linkcolor=blue,anchorcolor=blue, citecolor=blue,urlcolor=blue]{hyperref}

\newtheorem{theorem}{Theorem}[section]
\newtheorem{definition}[theorem]{Definition}
\newtheorem{lemma}[theorem]{Lemma}
\newtheorem{example}[theorem]{Example}
\newtheorem{corollary}[theorem]{Corollary}
\newtheorem{proposition}[theorem]{Proposition}

\newcommand{\R}{\mathbb{R}}
\newcommand{\RS}{\mathbb{R}_{\textnormal{s}}}
\newcommand{\C}{\mathbb{C}}
\newcommand{\CS}{\mathbb{C}_{\textnormal{s}}}
\newcommand{\PS}{\mathbb{C}_{\textnormal{ps}}}
\newcommand{\CPS}{\mathbb{C}_{\textnormal{cps}}}
\newcommand{\I}{\mathbb{I}}

\newcommand{\E}{\mathcal{E}}

\newcommand{\U}{\mathcal{U}}
\newcommand{\V}{\mathcal{V}}
\newcommand{\W}{\mathcal{W}}
\newcommand{\X}{\mathcal{X}}
\newcommand{\Y}{\mathcal{Y}}
\newcommand{\Z}{\mathcal{Z}}
\newcommand{\HI}{\mathcal{H}}
\newcommand{\SI}{\mathcal{S}}
\newcommand{\TT}{\mathcal{T}}
\newcommand{\OO}{\mathcal{O}}
\newcommand{\ba}{\boldsymbol{a}}
\newcommand{\bb}{\boldsymbol{b}}
\newcommand{\bc}{\boldsymbol{c}}
\newcommand{\be}{\boldsymbol{e}}

\newcommand{\bp}{\boldsymbol{p}}
\newcommand{\bs}{\boldsymbol{s}}
\newcommand{\bx}{\boldsymbol{x}}
\newcommand{\by}{\boldsymbol{y}}
\newcommand{\bz}{\boldsymbol{z}}
\newcommand{\bu}{\boldsymbol{u}}
\newcommand{\bv}{\boldsymbol{v}}

\newcommand{\od}{\otimes d}
\newcommand{\ov}{\overline}

\newcommand{\ii}{\mbox{\bf i}}
\newcommand{\T}{\textnormal{T}}
\newcommand{\HH}{\textnormal{H}}
\newcommand{\ex}{\textnormal{E}}
\newcommand{\even}{\textnormal{even}}

\newcommand{\re}{\textnormal{Re}\,}
\newcommand{\im}{\textnormal{Im}\,}
\newcommand{\tr}{\textnormal{tr}\,}

\newcommand{\rank}{\textnormal{rank}\,}
\newcommand{\ps}{\textnormal{PS}}
\newcommand{\cps}{\textnormal{CPS}}

\DeclareMathOperator*{\argmax}{arg\,max}
\DeclareMathOperator*{\argmin}{arg\,min}

\begin{document}

\title{On decompositions and approximations of conjugate partial-symmetric complex tensors}

\author{
Taoran FU
\thanks{School of Mathematical Sciences, Shanghai Jiao Tong University, Shanghai 200240, China. Email: \mbox{taoran30@sjtu.edu.cn}.}
    \and
Bo JIANG
\thanks{Research Institute for Interdisciplinary Sciences, School of Information Management and Engineering, Shanghai University of Finance and Economics, Shanghai 200433, China. Email: \mbox{isyebojiang@gmail.com}.}
    \and
Zhening LI
\thanks{Department of Mathematics, University of Portsmouth, Portsmouth, Hampshire PO1 3HF, United Kingdom. Email: \mbox{zheningli@gmail.com}.}
}

\date{\today}

\maketitle

\begin{abstract}

  Conjugate partial-symmetric (CPS) tensors are the high-order generalization of Hermitian matrices. As the role played by Hermitian matrices in matrix theory and quadratic optimization, CPS tensors have shown growing interest recently in tensor theory and optimization, particularly in many application-driven complex polynomial optimization problems.
  In this paper, we study CPS tensors with a focus on ranks, rank-one decompositions and approximations, as well as their applications. The analysis is conducted along side with a more general class of complex tensors called partial-symmetric tensors. We prove constructively that any CPS tensor can be decomposed into a sum of rank-one CPS tensors, which provides an alternative definition of CPS tensors via linear combinations of rank-one CPS tensors. Three types of ranks for CPS tensors are defined and shown to be different in general. This leads to the invalidity of the conjugate version of Comon's conjecture. We then study rank-one approximations and matricizations of CPS tensors. By carefully unfolding CPS tensors to Hermitian matrices, rank-one equivalence can be preserved. This enables us to develop new convex optimization models and algorithms to compute best rank-one approximation of CPS tensors. Numerical experiments from various data are performed to justify the capability of our methods.

%

\vspace{0.25cm}

\noindent {\bf Keywords:} conjugate partial-symmetric tensor, partial-symmetric tensor, rank, rank-one decomposition, rank-one approximation, tensor eigenvalue

\vspace{0.25cm}

\noindent {\bf Mathematics Subject Classification:}
  15A69,  
  15B57,  
  15A18,  
  15A03   

\end{abstract}

\section{Introduction}\label{sec:intro}

As the complex counterpart of real symmetric matrices, Hermitian matrices are often considered to play a more important role than complex symmetric matrices in practice. This is mainly due to the fact that any complex quadratic form generated by a Hermitian matrix always takes real values and all the eigenvalues of a Hermitian matrix are real. It is especially important in many applications, for instance, in quantum physics where Hermitian matrices are operators that measure properties of a system, e.g., total spin which has to be real, and in mathematical optimization whereas objective functions need to be real-valued. Generalizing to high-order tensors, symmetric tensors, no matter in the real or in the complex field, have been paid enormous attention in the recent decade. However, the high-order generalization of Hermitian matrices has not been formally proposed until recently by Jiang et al.~\cite{JLZ16}, who named it as {\em conjugate partial-symmetric} (CPS) tensors. Nevertheless, various examples of CPS tensors can be extracted from real applications in forms of complex polynomial optimization problems. Aittomaki and Koivunen~\cite{AK09} considered the beampattern optimization and formulated it as a complex multivariate quartic minimization model. Aubry et al.~\cite{ADJZ13} modeled a radar signal processing problem by optimizing a complex quartic polynomial which always takes real values.
Josz~\cite{J16} investigated applications of complex polynomial optimization to electricity transmission network. Moreover, Madani et al.~\cite{MLB16} studied the power system state estimation via complex polynomial optimization.

There were several discussions on high-order generalization of Hermitian matrices earlier and recently. The Hermitian tensor product, defined to be the Kronecker product of a Hermitian matrix, has been studied since 1960s~\cite{MM64,K73}. It has many applications in quantum entanglement~\cite{FJS06} and enjoys certain nice properties, such as the Kronecker product of a Hermitian matrix remains a Hermitian matrix. Ni et al.~\cite{NQB14} proposed the unitary eigenvalues and unitary symmetric eigenvalues for complex tensors and symmetric complex tensors, respectively, and demonstrated a relation to the geometric measure of quantum entanglement. Jiang et al.~\cite{JLZ16} characterized real-valued complex polynomial functions and their symmetric tensor representations, which naturally led to the definition of CPS tensors as well as its generalization called conjugate super-symmetric tensors. Eigenvalues and applications for these tensors were discussed as well. Recently, Derksen et al.~\cite{DFLW17} studied entanglement of $d$-partite system in the field of quantum mechanics and introduced the notion of bisymmetric Hermitian tensor, which is essentially the same definition to CPS tensors in~\cite{JLZ16}. Some elementary properties of bisymmetric Hermitian tensors were discussed.

CPS tensors indeed inherit many nice properties of Hermitian matrices. For instance, every symmetric complex form generated by a CPS tensor is real-valued and all the eigenvalues of a CPS tensor are real~\cite{JLZ16}. However, in contrast to the many great efforts on the computational aspect~\cite{ZH06,SZY07,HZ10,JLZ14,FJL18,FF18} of CPS tensors, the analysis of their theoretical properties is still limited and less developed. One particular aspect is decompositions and ranks, which are believed to be one of the most important topics for high-order tensors. This is what the current paper aims for.
As we all know that the generalization of matrices to high-order tensors has led to interesting new findings as well as keeping many nice properties, CPS tensors, as a generalization of Hermitian matrices in terms of the high order and a generalization of real symmetric tensors in terms of the complex field, should also be expected to behave in that sense.
One of our findings states that Comon's type conjecture, i.e., the symmetric rank of a symmetric tensor is equal to the rank of the tensor, applies to CPS tensors is actually invalid at a simple example. We believe the theoretical analysis along this line will provide novel insights into CPS tensors, and hope these new findings will help in future modelling of practical applications. In fact, one of our results on rank-one equivalence via matrix unfolding helps to develop new models and algorithms to compute the best rank-one approximation and the extreme eigenvalue of CPS tensors.

The study of CPS tensors in this paper is focus on ranks, rank-one decompositions and approximations, as well as their applications. The analysis is conducted along side with a more general class of complex tensors called partial-symmetric (PS) tensors. We propose the Hermitian and skew-Hermitian parts of PS tensors, which are helpful to understand structures of PS tensors and CPS tensors. We prove constructively that any CPS tensor can be decomposed into a sum of rank-one CPS tensors, using the tools in additive number theory, specifically, Hilbert's identity~\cite{B02, JHLZ14}. This provides an alternative definition of CPS tensors via real linear combinations of rank-one CPS tensors. We then define three types of ranks for CPS tensors, and show that they are non-identical in general. For CPS tensors, this leads to the invalidity of Comon's conjecture for CPS tensors, albeit it is not the exact form of Comon's conjecture in~\cite{CGLM08}. We further study rank-one approximations of CPS tensors. Depends on the types of rank-one tensors to be considered, rank-one approximations could also be different. As is known in the literature, if the square matricization of an even-order symmetric tensor is rank-one, then the original symmetric tensor is also rank-one. We figure out that the same property does hold for CPS tensors when they are unfolded to Hermitian matrices under a {\em careful way} of matricization. Based on this equivalence, we propose two convex optimization models for the best rank-one approximation of CPS tensors. Several numerical experiments from real data to simulated data are performed to justify the capability of our methods.

This paper is organized as follows. We start with the preparations of various notations, definitions, and elementary properties in Section~\ref{sec:prep}. In Section~\ref{sec:p1}, we study partial-symmetric decompositions of CPS tensors and PS tensors, and discuss several concepts of ranks for CPS tensors. We then focus on rank-one approximations and rank-one equivalence via matricization for CPS tensors in Section~\ref{sec:p2}, which provide an immediate application to find the best rank-one approximation of CPS tensors via convex optimization relaxations. Finally in Section~\ref{sec:numerical}, we conducted several numerical experiments to illustrate effective performance of the methods proposed in Section~\ref{sec:p2}.

\section{Preparations}\label{sec:prep}

Throughout this paper, we uniformly use the lowercase letters, boldface lowercase letters, capital letters, and calligraphic letters to denote scalars, vectors, matrices, and high-order tensors, respectively, e.g., a scalar $x$, a vector $\bx$, a matrix $X$, and a tensor $\X$. We use subscripts to denote their components, e.g., $x_i$ being the $i$-th entry of a vector $\bx$, $X_{ij}$ being the $(i,j)$-th entry of a matrix $X$, and $\X_{ijk}$ being the $(i,j,k)$-th entry of a third-order tensor $\X$. As usual, the field of real numbers and the field of complex numbers are denoted by $\R$ and $\C$, respectively.

For any complex number $z = x + \ii y\in\C$ with $x,y\in\R$, its real part and imaginary part are denoted by $\re z:= x$ and $\im z:= y$, respectively. Its argument is denoted by $\arg(z)$ and its modulus is denoted by $|z|: = \sqrt{\ov{z}z} = \sqrt{x^2 + y^2}$, where $\ov{z}: = x-\ii y$ denotes the conjugate of $z$. For any vector $\bz\in\C^n$, we denote $\bz^{\HH}: = \ov\bz^{\T}$ to be the transpose of its conjugate, and we define it analogously for matrices. The norm of a complex vector $\bz\in\C^n$ is defined as $\|\bz\|: = \sqrt{\bz^{\HH}\bz} = \sqrt{\sum_{i=1}^n |z_i|^2}$.

\subsection{Partial-symmetric tensor and conjugate partial-symmetric tensor}

We consider the space of cubic complex tensors of dimensional $n$ and order $d$, denoted by $\C^{n^d}$. A tensor $\TT\in\C^{n^d}$ is called symmetric if every entry of $\TT$ is invariant under all permutations of its indices, i.e., for every $1\le i_1\le\dots\le i_{d}\le n$,
$$
\TT_{j_1\dots j_d} = \TT_{i_1\dots i_d} \quad \forall\, (j_1,\dots, j_d) \in \Pi (i_1,\dots, i_d),
$$
where $\Pi(i_1,\dots, i_d)$ denotes the set of all distinctive permutations of $\{i_1,\dots,i_d\}$. The set of symmetric tensors in $\C^{n^d}$ is denoted by $\CS^{n^d}$.

\begin{definition} {\em (Jiang et al.~\cite[Definition 2.3]{JLZ16}).} \label{def:ps}
An even-order complex tensor $\TT\in\C^{n^{2d}}$ is called partial-symmetric (PS) if for every $1\le i_1\le\dots\le i_{d}\le n$ and $1\le i_{d+1}\le\dots\le i_{2d}\le n$,
\begin{equation*}
\TT_{j_1\dots j_d j_{d+1}\dots j_{2d}}  =  \TT_{i_1\dots i_d i_{d+1}\dots i_{2d}} \quad\forall \,
(j_1,\dots, j_d) \in \Pi (i_1,\dots, i_d),\, (j_{d+1},\dots, j_{2d}) \in\Pi(i_{d+1},\dots, i_{2d}).
\end{equation*}
\end{definition}
Essentially, a PS tensor is symmetric with respect to its first half modes, and also symmetric with respect to its last half modes, while a symmetric tensor is symmetric with respect to all its modes. The set of PS complex tensors in $\C^{n^{2d}}$ is denoted by $\PS^{n^{2d}}$. When $d=1$, one has $\CS^{n^2}\subsetneq\PS^{n^2}=\C^{n^2}$. However, for $d\ge2$, it is obvious that $\CS^{n^{2d}}\subsetneq\PS^{n^{2d}}\subsetneq \C^{n^{2d}}$.

A special class of PS tensors, called conjugate partial-symmetric tensors, generalizes Hermitian matrices to high-order tensor spaces.
\begin{definition} {\em (Jiang et al.~\cite[Definition 3.7]{JLZ16}).} \label{def:cps}
An even-order complex tensor $\TT\in\C^{n^{2d}}$ is called conjugate partial-symmetric (CPS) if it is partial-symmetric and
$$
\TT_{i_1\dots i_d i_{d+1}\dots i_{2d}} = \ov{\TT_{i_{d+1}\dots i_{2d} i_1\dots i_d}} \quad \forall\, 1\le i_1\le\dots\le i_{d}\le n,\,1\le i_{d+1}\le\dots\le i_{2d}\le n.
$$
\end{definition}
The set of CPS tensors in $\C^{n^{2d}}$ is denoted by $\CPS^{n^{2d}}$. Obviously when $d=1$, $\CPS^{n^{2}}$ is nothing but the set of Hermitian matrices in $\C^{n^2}$. CPS tensors are the high-order generalization of Hermitian matrices. For $d\ge2$, one has $\CPS^{n^{2d}}\subsetneq\PS^{n^{2d}}\subsetneq \C^{n^{2d}}$. However, $\CS^{n^{2d}}$ and $\CPS^{n^{2d}}$ are not comparable, and actually $\CS^{n^{2d}} \cap\CPS^{n^{2d}} = \RS^{n^{2d}}$, the set of real symmetric tensors. This is the the same for complex symmetric matrices and Hermitian matrices.

We remark that PS tensors are closed under addition and multiplication by complex numbers, while CPS tensors are closed under addition and multiplication by real numbers only. This fact may not be obvious from their definitions. In fact, it can be easily seen from their equivalent definitions via partial-symmetric decompositions; see Section~\ref{sec:psdec}.

\subsection{Complex form}

The Frobenius inner product of two complex tensors $\U,\V\in\C^{n^d}$ are defined as
$$
\langle \U, \V \rangle := \sum_{i_1=1}^n\dots\sum_{i_d=1}^n \ov{\U_{i_1\dots i_d}} \V_{i_1\dots i_d},
$$
and its induced Frobenius norm of a complex tensor $\TT$ is naturally defined as
$
\|\TT\|:=\sqrt{\left\langle \TT,\TT\right\rangle}.
$
We remark that these two notations naturally apply to vectors and matrices, which are tensors of order one and order two, respectively. A rank-one tensor, also called a simple tensor, is a tensor that can be written as outer products of vectors, i.e., $\bx_1\otimes\dots\otimes\bx_d\in\C^{n^d}$ where $\bx_k\in\C^n$ for $k=1,\dots,d$.

Given a complex tensor $\TT\in\C^{n^d}$, the multilinear form of $\TT$ is defined as
\begin{align} \label{eq:multilinear}
  \TT(\bx_1,\dots,\bx_d):=\sum_{i_1=1}^n\dots\sum_{i_d=1}^n \TT_{i_1\dots i_d}(x_1)_{i_1}\dots (x_d)_{i_d}
   = \langle \ov\TT, \bx_1\otimes \dots \otimes \bx_d \rangle,
\end{align}
where the variable $\bx_k\in\C^n$ for $k=1,\dots,d$.

If a vector in a multilinear form~\eqref{eq:multilinear} is missing and replaced by a `$\bullet$', say $\TT(\bullet,\bx_2,\dots,\bx_d)$, it then becomes a vector in $\C^n$. Explicitly, the $i$-th entry of $\TT(\bullet,\bx_2,\dots,\bx_d)$ is $\TT(\be_i,\bx_2,\dots,\bx_d)$ for $i=1,\dots,n$, where $\be_i$ is the $i$-th unit vector of $\R^n$. When all $\bx_k$'s in~\eqref{eq:multilinear} are the same, a multilinear form becomes a complex homogenous polynomial function (or complex form) of $\bx\in\C^n$, i.e.,
\begin{align*}
 \TT(\bx^d):= \TT(\underbrace{\bx, \dots, \bx}_d)
   = \langle \ov\TT, \underbrace{\bx \otimes \dots \otimes \bx}_d \rangle
   =: \langle \ov\TT, \bx^{\od} \rangle.
\end{align*}
The notations, $\bx^d$ standing for $d$ copies of $\bx$ in a multilinear form, and $\bx^{\od}$ standing for outer products of $d$ copies of $\bx$, will be used throughout this paper as long as there is no ambiguity.

To our particular interest in this paper, the following conjugate complex form, or conjugate form, is defined by a PS tensor $\TT\in\PS^{n^{2d}}$,
\begin{align*}
 \TT({\ov\bx}^d\bx^d):= \TT(\underbrace{\ov\bx,\dots,\ov\bx}_d, \underbrace{\bx, \dots, \bx}_d)
= \langle \ov\TT, \ov\bx^{\od}\otimes \bx^{\od} \rangle.
\end{align*}
Remark that ${\ov\bx}^d$ and $\bx^d$ in $\TT({\ov\bx}^d\bx^d)$ cannot be swapped as otherwise it becomes a different form. Similarly, we may use the notation $\TT(\bullet\,{\ov\bx}^{d-1}\bx^d)\in\C^n$, which equals $\TT({\ov\bx}^{d-1}\bullet\bx^d)$ since $\TT$ is a PS tensor.

\subsection{Hermitian part and skew-Hermitian part}

It is shown in~\cite{JLZ16} that $\TT({\ov\bx}^d\bx^d)$ is real-valued if and only if $\TT$ is CPS, extending the case of $d=1$, i.e., $A(\ov\bx\bx)$ is real-valued if and only if $A$ is a Hermitian matrix. As is well known, any complex matrix $A\in \C^{n^2}$ can be written as $A= H(A) + S(A)$, where
$$H(A) = \frac{1}{2}(A + A^{\HH}) \mbox{ and } S(A) = \frac{1}{2}(A - A^{\HH})$$
are the Hermitian part and the skew-Hermitian part of $A$, respectively. We extend this concept to high-order PS tensors, which is helpful in the analysis of our results.
\begin{definition} \label{def:hs}
The conjugate transpose of a PS tensor $\TT\in\PS^{n^{2d}}$, denoted by $\TT^{\HH}$, satisfies
$$(\TT^{\HH})_{i_{1}\dots i_{d}i_{d+1}\dots i_{2d}} = \ov{\TT_{i_{d+1}\dots i_{2d}i_{1}\dots i_{d}}} \quad\forall\,
1\le i_{1}\le \dots \le i_{d}\le n,\,1\le i_{d+1}\le \dots \le i_{2d}\le n.
$$
The Hermitian part $\HI(\TT)$ and the skew-Hermitian part $\SI(\TT)$ of a PS tensor $\TT$ are defined as
$$
\HI(\TT) := \frac{1}{2}(\TT + \TT^{\HH}) \mbox{ and } \SI(\TT) := \frac{1}{2}(\TT-\TT^{\HH}),
$$
respectively.
\end{definition}

Obviously, one has $(\TT^{\HH})^{\HH}=\TT$ for a PS tensor $\TT$. It is clear from Definition~\ref{def:cps} that a PS tensor $\TT$ is CPS if and only if $\TT^{\HH}=\TT$, or $\SI(\TT)=\OO$, the zero tensor. The following property can be verified straightforwardly, similar to Hermitian matrices.
\begin{proposition}\label{thm:parts}
  Any PS tensor $\TT$ can be uniquely written as $\TT = \U + \ii \V$, where both $\U$ and $\V$ are CPS tensors. In particular, $\U=\HI(\TT)$ and $\V=-\ii \SI(\TT)$. Moreover, $\TT$ is a CPS tensor if and only if $\V=\OO$.
\end{proposition}
\begin{proof}
Obviously $\TT=\HI(\TT)+\ii(-\ii \SI(\TT))$. According to Definition~\ref{def:hs}, we have
\begin{align*}
\HI(\TT)^{\HH}&=\frac{1}{2}(\TT + \TT^{\HH})^{\HH}=\frac{1}{2}( \TT^{\HH}+\TT)= \HI(\TT),\\
(-\ii \SI(\TT))^{\HH} &= \left(-\frac{\ii}{2}(\TT-\TT^{\HH})\right)^{\HH}= \frac{\ii}{2}(\TT^{\HH}-\TT)
= -\ii \SI(\TT),
\end{align*}
implying that both $\U=\HI(\TT)$ and $\V=-\ii \SI(\TT)$ are CPS.

For the uniqueness, suppose that $\TT=\X+\ii \Y$ with $\X,\Y$ being CPS. We have
$$
\U=\HI(\TT)=\frac{\TT+\TT^{\HH}}{2}= \frac{\X+\ii \Y + (\X+\ii \Y)^{\HH}}{2} =\frac{\X+\ii \Y + \X^{\HH}-\ii \Y^{\HH}}{2}  = \X.
$$
Using a similar argument, one can show that $\V=\Y$.

If $\TT$ is a CPS tensor, then $$\U+\ii\V=\TT=\TT^{\HH}=\U^{\HH}-\ii\V^{\HH}=\U-\ii\V,$$
implies that $\V=\OO$. If $\V=\OO$, then obviously $\TT$ is a CPS tensor.
\end{proof}

\section{Partial-symmetric decomposition and rank}\label{sec:p1}

This section is devoted to decompositions of PS tensors and CPS tensors. It is an extension of the symmetric
decomposition of symmetric tensors. One main result is to propose a constructive approach to decompose a CPS tensor into a sum of rank-one CPS tensors, and hence provide an alternative definition of CPS tensors via real linear combination of rank-one CPS tensors. Based on these results, we discuss several ranks for PS tensors and CPS tensors, which can be classified as the conjugate version of Waring's decomposition~\cite{BCMT10}

\subsection{Partial-symmetric decomposition}\label{sec:psdec}

Rank-one decompositions play an essential role in exploring structures of high-order tensors. For Hermitian matrices, they enjoy the following type of partial-symmetric decompositions: If $A\in\CPS^{n^2}$, then
\begin{equation}\label{eq:hem}
A = \sum_{j = 1}^r \lambda_j\ov{\ba_j}{\ba_j}^{\T} = \sum_{j = 1}^r \lambda_j\ov{\ba_j}\otimes \ba_j,
\end{equation}
where $\lambda_j\in\R$ and $\ba_j\in\C^n$ for $j = 1,\dots,r$ with some $r\le n$. As the high-order generalization of Hermitian matrices, CPS tensors do inherit this important property. Before presenting the main result in this section, let us first prove a technical result.

\begin{lemma} \label{thm:hilbert}
For any $\bx\in\C^n$, there exists $\ba_j\in\C^n$ and $\lambda_j\in\R$ for $j=1,\dots,m$ with finite $m$, such that
$$\left|\sum_{i=1}^n{x_i}^d\right|^2 =\sum_{j=1}^m\lambda_j \left|{\ba_j}^{\T}\bx\right|^{2d}.$$
\end{lemma}
\begin{proof}
For any nonzero $\alpha_0, \alpha_1, \dots , \alpha_{2d}\in\R$ with $\alpha_i\ne\alpha_j$ if $i\ne j$, consider the following $2d+1$ linear equations with $2d+1$ variables:
\begin{equation}\label{eq:van1}
{\alpha_0}^k z_0 + {\alpha_1}^k z_1 +  \dots  +{\alpha_{2d}}^k z_{2d} = \gamma_k\quad k=0,1,\dots,2d,
\end{equation}
where $\gamma_0=1,\,\gamma_d=\sqrt{d!},\,\gamma_{2d}=d!$, and other $\gamma_k$'s are zeros.
The determinant of the coefficients of~\eqref{eq:van1} is the Vandermonde determinant
$$
\left | \begin{array}{cccc}
1 & 1 & \dots & 1 \\
\alpha_0 & \alpha_1 & \dots &\alpha_{2d}   \\
\vdots &\vdots   &\vdots & \vdots \\
{\alpha_0}^{2d}  & {\alpha_1}^{2d}  & \dots &{\alpha_{2d}}^{2d}\\
\end{array} \right | = \prod_{0 \le i < j \le d} (\alpha_j - \alpha_i) \ne 0,
$$
and so~\eqref{eq:van1} has a unique real solution, which is denoted by $(z_0, z_1, \dots, z_{2d})$ for simplicity.

Denote $\Omega_{d}:=\{e^{\ii\frac{2k\pi}{d}}:k=0,1,\dots,d-1\}$, the set of all complex solutions to $z^d=1$. Let $\xi_1,\dots, \xi_n$ be i.i.d.\ random variables uniformly distributed on $\Omega_{d}$. For any linear function of $\xi_i$'s, say, $c_1\xi_1+\dots + c_n\xi_n$, it follows that
\begin{equation}\label{eq:expand}
\ex \left[ \left(\ov{ c_1\xi_1+\dots + c_n\xi_n}\right)^d \left( c_1\xi_1+\dots + c_n\xi_n\right)^d \right]=
\sum_{d_1+\dots+d_n=d}\left( \frac{d!}{d_1 ! \dots d_n !} \right)^2 \prod_{i=1}^{n} \left| c_i\right|^{2d_i}
+ \sum_{i \ne j} \ov{c_i}^d {c_j}^d.
\end{equation}
To see why~\eqref{eq:expand} holds, we consider all terms $\prod_{i=1}^n\left(\ov{c_i\xi_i}\right)^{d_i}\left(c_i\xi_i\right)^{t_i}$ with $d_1+\dots+d_n=t_1+\dots+t_n=d$ in the left hand side of~\eqref{eq:expand}. They can be classified into three mutually exclusive cases: (i) If there is some $i$ such that $1\le|d_i-t_i|\le d-1$, then
$$
\ex\left[\prod_{i=1}^n\left(\ov{c_i\xi_i}\right)^{d_i}\left(c_i\xi_i\right)^{t_i}\right] =
\ex\left[\ov{c_i}^{d_i}{c_i}^{t_i}\ov{\xi_i}^{d_i}{\xi_i}^{t_i}\right]
\ex\left[\prod_{j\ne i}\left(\ov{c_j\xi_j}\right)^{d_j}\left(c_j\xi_j\right)^{t_j}\right]
 = 0
$$
since $\ex[{\xi_i}^k]=0$ for any integer $k$ with $1\le |k|\le d-1$; (ii) If there are two indices $i\ne j$ such $d_i=t_j=d$, then
$$
\ex\left[\prod_{i=1}^n\left(\ov{c_i\xi_i}\right)^{d_i}\left(c_i\xi_i\right)^{t_i}\right] =
\ex\left[\left(\ov{c_i\xi_i}\right)^d\left(c_j\xi_j\right)^d\right] =\ov{c_i}^d {c_j}^d;
$$
and (iii) If $d_i=t_i$ for all $1\le i\le n$, then
$$
\ex\left[\prod_{i=1}^n\left(\ov{c_i\xi_i}\right)^{d_i}\left(c_i\xi_i\right)^{t_i}\right] = \prod_{i=1}^n\ex\left[\left(\ov{c_i\xi_i}\right)^{d_i}\left(c_i\xi_i\right)^{d_i}\right] = \prod_{i=1}^{n}\left| c_i\right|^{2d_i}
$$
and the number of such terms in~\eqref{eq:expand} is $\left( \frac{d!}{d_1 ! \dots d_n !} \right)^2$.

By applying~\eqref{eq:expand} and~\eqref{eq:van1}, we obtain
\begin{align}
&~~~~\frac{1}{d!} \sum_{k_1 = 0}^{2d} \dots \sum_{k_n = 0}^{2d} \left(\prod_{\ell=1}^nz_{k_\ell}\right)  \ex \left[ \left(\ov{ \alpha_{k_1}\xi_1x_1 + \dots + \alpha_{k_n}\xi_n x_n}\right)^d \left( \alpha_{k_1} \xi_1 x_1 + \dots + \alpha_{k_n} \xi_n x_n\right)^d \right] \nonumber \\
&= \frac{1}{d!}  \sum_{k_1 = 0}^{2d} \dots \sum_{k_n = 0}^{2d} \left(\prod_{\ell=1}^nz_{k_\ell}\right) \left( \sum_{i \ne j} \ov{\alpha_{k_i}}^d \ov{x_i}^d {\alpha_{k_j}}^d{x_j}^d + \sum_{d_1+ \dots+ d_n=d} \left( \frac{d!}{d_1 ! \dots d_n !} \right)^2 \prod_{i=1}^{n} \left| \alpha_{k_i}x_i \right|^{2d_i}
  \right) \nonumber \\
& = \frac{1}{d!} \sum_{i \ne j} \sum_{k_1 = 0}^{2d} \dots \sum_{k_n = 0}^{2d} \left(\prod_{\ell=1}^nz_{k_\ell}\right)  \ov{\alpha_{k_i}}^d {\alpha_{k_j}}^d \ov{x_i}^d{x_j}^d  + d! \sum_{d_1+ \dots+ d_n=d}  \sum_{k_1 = 0}^{2d} \dots \sum_{k_n = 0}^{2d} \prod_{i=1}^{n}  \frac{\left|x_i \right|^{2d_i}}{\left(d_i !\right)^2} \left( {\alpha_{k_i}} ^{2d_i} z_{k_i} \right)  \nonumber \\
&= \frac{1}{d!} \sum_{i \ne j} \ov{x_i}^d{x_j}^d  \left(\sum_{k=0}^{2d} {\alpha_k}^{d} z_k \right)^2 \left(\sum_{k=0}^{2d} z_k \right)^{n-2} + d!\sum_{d_1+ \dots+ d_n=d} \prod_{i=1}^{n} \frac{\left|x_i \right|^{2d_i}}{\left(d_i !\right)^2} \left(\sum_{k=0}^{2d} {\alpha_k}^{2d_i} z_k \right) \nonumber \\
&=  \sum_{i \ne j} \ov{x_i}^d{x_j}^d + \sum_{i =1}^{n} |x_i|^{2d} + \even(d) \frac{\left(d!\right)^2}{\left(\left(d/2\right)!\right)^4}\sum_{i < j}  |x_i|^d  |x_j|^d \nonumber \\
&=  \left|\sum_{i=1}^n {x_i}^d \right |^2 + \even(d) \frac{\left(d!\right)^2}{\left(\left(d/2\right)!\right)^4}\sum_{i < j}  |x_i|^d  |x_j|^d, \label{eq:oddevenexp}
\end{align}
where $\even(d)$ is one if $d$ is even and zero otherwise. Notice that $\frac{1}{d!}\left(\prod_{\ell=1}^nz_{k_\ell}\right)$ is real, and so~\eqref{eq:oddevenexp} provides a constructive expression of $\sum_{j=1}^m\lambda_j \left|{\ba_j}^{\T}\bx\right|^{2d}$ with finite $m$ for $\left|\sum_{i=1}^n {x_i}^d \right |^2$ when $d$ is odd.

When $d$ is even, we consider another system of $d+1$ linear equations with $d+1$ variables:
\begin{equation}\label{eq:van2}
{\beta_0}^{2k} y_0 + {\beta_1}^{2k} y_1 +  \dots  +{\beta_d}^{2k} y_{d} = \delta_k\quad k=0,1,\dots,d,
\end{equation}
where nonzero $\beta_0, \beta_1, \dots , \beta_{d}\in\R$ with ${\beta_i}^2\ne{\beta_j}^2$ if $i\ne j$, and $\delta_0=\delta_{d/2}=1$ and other $\delta_k$'s are zeros.
Similar to the linear system~\eqref{eq:van1} whose Vandermonde determinant is nonzero,~\eqref{eq:van2} also has a unique real solution, denoted by $(y_0, y_1, \dots, y_d)$ for simplicity.

Let $\eta_1,\dots, \eta_n$ be i.i.d.\ random variables uniformly distributed on $\Omega_{d+1}$. Similar to the proof of~\eqref{eq:expand}, it is easy to obtain
$$
\ex \left[ \left(\ov{ c_1\eta_1+\dots + c_n\eta_n}\right)^d \left( c_1\eta_1+\dots + c_n\eta_n\right)^d \right]=
\sum_{d_1+\dots+d_n=d}\left( \frac{d!}{d_1 ! \dots d_n !} \right)^2 \prod_{i=1}^{n} \left| c_i\right|^{2d_i}.
$$
Therefore,
\begin{align*}
&~~~~\sum_{k_1 = 0}^{d} \dots \sum_{k_n = 0}^{d} \left(\prod_{\ell=1}^ny_{k_\ell}\right) \ex\left[ \left(\ov{\beta_{k_1} \eta_1x_1 + \dots + \beta_{k_n}\eta_nx_n}\right)^d \left(\beta_{k_1}\eta_1x_1 + \dots + \beta_{k_n}\eta_nx_n\right)^d \right ]\\
&= \sum_{k_1 = 0}^{d} \dots \sum_{k_n = 0}^{d} \left(\prod_{\ell=1}^ny_{k_\ell}\right) \sum_{d_1+ \dots+ d_n=d} \left( \frac{d!}{d_1! \dots d_n!} \right)^2 \prod_{i=1}^{n} \left| \beta_{k_i}x_i \right|^{2d_i}\\
&= (d!)^2 \sum_{d_1+ \dots+ d_n=d} \sum_{k_1 = 0}^{d} \dots \sum_{k_n = 0}^{d} \prod_{i=1}^n\frac{\left|x_i \right|^{2d_i}}{(d_i!)^2} \left({\beta_{k_i}}^{2d_i}y_{k_i} \right)\\
&= (d!)^2 \sum_{d_1+ \dots+ d_n=d} \prod_{i=1}^{n}  \frac{\left|x_i \right|^{2d_i}}{(d_i !)^2} \left(\sum_{k=0}^d {\beta_k} ^{2d_i} y_k \right) \\
&= \sum_{i<j} \frac{(d!)^2}{((d/2)!)^4}|x_i|^d|x_j|^d,
\end{align*}
where the last inequality is due to~\eqref{eq:van2}. This, together with~\eqref{eq:oddevenexp} for even $d$, gives that
\begin{align*}
\left | \sum_{i=1}^n x_i^d \right |^2 &=\frac{1}{d!} \sum_{k_1 = 0}^{2d} \dots \sum_{k_n = 0}^{2d} \left(\prod_{\ell=1}^nz_{k_\ell}\right)  \ex \left [ \left(\ov{ \alpha_{k_1} \xi_1x_1 + \dots + \alpha_{k_n} \xi_n x_n}\right)^d \left( \alpha_{k_1} \xi_1 x_1 + \dots + \alpha_{k_n} \xi_n x_n\right)^d \right ]\\
& ~~~~- \sum_{k_1 = 0}^{d} \dots \sum_{k_n = 0}^{d} \left(\prod_{\ell=1}^ny_{k_\ell}\right)  \ex \left [ \left(\ov{ \beta_{k_1} \eta_1x_1 + \dots + \beta_{k_n} \eta_n x_n}\right)^d \left( \beta_{k_1} \eta_1 x_1 + \dots + \beta_{k_n} \eta_n x_n\right)^d \right]
\end{align*}
providing an expression of $\sum_{j=1}^m\lambda_j \left|{\ba_j}^{\T}\bx\right|^{2d}$ with finite $m$ for $\left|\sum_{i=1}^n {x_i}^d \right |^2$ when $d$ is even.
\end{proof}
The above lemma can be viewed in some sense as a complex type of Hilbert's identity in the literature (see e.g.,~\cite{B02}), which states that for any positive integers $d$ and $n$, there always exist $\ba_1,\dots,\ba_m\in\R^n$, such chat
$(\bx^{\T}\bx)^d = \sum_{j=1}^m ({\ba_j}^{\T}\bx)^{2d}$.

With Lemma~\ref{thm:hilbert} in hand, we are ready to show that any CPS tensor can be decomposed into a sum of rank-one CPS tensors.
\begin{theorem}\label{thm:cps}
An even-order tensor $\TT \in \C^{n^{2d}}$ is CPS if and only if $\TT$ has the following partial-symmetric decomposition
\begin{equation}\label{eq:cpsd}
\TT = \sum_{j = 1}^{m}\lambda_j \ov{\ba_j}^{\od}  \otimes {\ba_j}^{\od},
\end{equation}
where $\lambda_j \in \R$ and $\ba_j\in \C^{n}$ for $j=1,\dots,m$ with finite $m$.
\end{theorem}
\begin{proof}
For any $\ba\in\C^n$, it is straightforward to check that the rank-one tensor $\ov\ba^{\od}\otimes\ba^{\od}$ is CPS by Definition~\ref{def:cps}. In fact, it is symmetric with respect to its first half modes, and also symmetric with respect to its last half modes, resulting a PS tensor. Besides, this PS tensor satisfies
$$
({\ov\ba}^{\od}\otimes\ba^{\od})^{\HH} = \ov{\ba^{\od}}\otimes\ov{{\ov\ba}^{\od}} = {\ov\ba}^{\od}\otimes\ba^{\od},
$$
resulting a CPS tensor. Therefore, as a real linear combination of such rank-one CPS tensors in~\eqref{eq:cpsd}, $\TT$ must also be CPS.

On the other hand, according to~\cite[Proposition 3.9]{JLZ16}, any CPS tensor $\TT$ can be written as
\begin{equation} \label{eq:decom}
\TT = \sum_{j = 1}^q\lambda_j\ov{\Z_j}\otimes\Z_j
\end{equation}
where $\lambda_j\in \{-1,1\}$ and $\Z_j\in\CS^{n^{d}}$ for $j=1,\dots,q$ with finite $q$. It suffices to prove that $\ov\Z\otimes\Z$ admits a decomposition of~\eqref{eq:cpsd} if $\Z\in\CS^{n^d}$.

Since any symmetric complex tensor admits a finite symmetric decomposition (see e.g.,~\cite{CGLM08}), we may let $\Z=\sum_{j=1}^r{\ba_j}^{\od}$ where $\ba_j\in\C^n$ for $j=1,\dots,r$ with finite $r$. For any $\bx\in \C^{n}$, one has
\begin{align*}
(\ov{\Z}\otimes\Z)(\ov\bx^d\bx^d)
=\langle \Z\otimes\ov\Z, \ov\bx^{\od}\otimes \bx^{\od} \rangle
=\langle \Z, \ov\bx^{\od} \rangle \cdot \langle \ov\Z, \bx^{\od} \rangle
=\ov{\Z}(\ov\bx^d)\Z(\bx^d)
=|\Z(\bx^d)|^2
\end{align*}
and
\begin{align*}
\Z(\bx^d)=\langle \ov\Z, \bx^{\od}\rangle = \left\langle \sum_{j=1}^r\ov{\ba_j}^{\od}, \bx^{\od} \right\rangle = \sum_{j=1}^r \langle\ov{\ba_j}^{\od}, \bx^{\od}\rangle
= \sum_{j = 1}^r ({\ba_j}^{\T}\bx)^{d} = \sum_{j = 1}^r{y_j}^{d},
\end{align*}
where we let $\by=A\bx\in\C^r$ and $A=(\ba_1,\dots,\ba_r)^{\T}\in\C^{r\times n}$. Therefore, by Lemma~\ref{thm:hilbert}, there exist $\alpha_k\in\R$ and $\bb_k\in\C^r$ for $k=1,\dots,s$ with finite $s$, such that
$$
(\ov{\Z}\otimes\Z)(\ov\bx^d\bx^d)=|\Z(\bx^d)|^2
 =\left|\sum_{j = 1}^r{y_j}^{d}\right|^{2}
 = \sum_{k = 1}^{s}\alpha_k\left|{\bb_k}^{\T}\by\right|^{2d}
 = \sum_{k = 1}^{s}\alpha_k\left|{\bb_k}^{\T}A\bx\right|^{2d}.
$$
Finally, by letting $\bc_k=A^{\T}\bb_k\in\C^n$ for $k=1,\dots,s$, we obtain
\begin{align*}
(\ov{\Z}\otimes\Z)(\ov\bx^d\bx^d)
 = \sum_{k = 1}^{s} \alpha_k \left|{\bc_k}^{\T}\bx\right|^{2d}
 = \sum_{k = 1}^{s} \alpha_k \langle {\bc_k}^{\od}\otimes \ov{\bc_k}^{\od} ,\ov\bx^{\od}\otimes \bx^{\od}\rangle
 = \left(\sum_{k = 1}^{s}\alpha_k\ov{\bc_k}^{\od}\otimes {\bc_k}^{\od}\right)(\ov\bx^d\bx^d).
\end{align*}
This implies that $\ov{\Z}\otimes\Z=\sum_{k = 1}^{s}\alpha_k\ov{\bc_k}^{\od}\otimes {\bc_k}^{\od}$, completing the whole proof.
\end{proof}

Theorem~\ref{thm:cps} provides an alternative definition of a CPS tensor via a real linear combination of rank-one CPS tensors, i.e.,
$$
\CPS^{n^{2d}}:=\left\{ \sum_{j = 1}^{m}\lambda_j \ov{\ba_j}^{\od} \otimes {\ba_j}^{\od}: \lambda_j\in\R,\,\ba_j\in\C^n,\, j=1,\dots,m\right\}.
$$
The proof of Theorem~\ref{thm:cps} actually develops an explicit algorithm to decompose a general CPS tensor into a sum of rank-one CPS tensors. The procedure involves the following three main steps:
\begin{enumerate}
  \item[(i)] Find $\TT = \sum_{j = 1}^q\lambda_j\ov{\Z_j}\otimes\Z_j$ where $\lambda_j\in \{-1,1\}$ and $\Z_j\in\CS^{n^{d}}$ is symmetric;
  \item[(ii)] Find a symmetric rank-one decomposition for every $\Z_j$, i.e., $\Z_j=\sum_{k=1}^{r_j}{\ba_{jk}}^{\od}$ where $\ba_{jk}\in\C^n$;
  \item[(iii)] Find $\ov{\sum_{k=1}^{r_j}{\ba_{jk}}^{\od}}\otimes\sum_{k=1}^{r_j}{\ba_{jk}}^{\od}=\sum_{k = 1}^{s_j}\alpha_{jk}\ov{\bc_{jk}}^{\od}\otimes {\bc_{jk}}^{\od}$ where $\alpha_{jk}\in\R$ and $\bc_{jk}\in\C^n$ for every $j$.
\end{enumerate}

In fact, PS tensors also enjoy similar decompositions, via {\em complex} linear combinations of rank-one CPS tensors.
\begin{corollary}\label{thm:ps}
An even-order tensor $\TT \in \C^{n^{2d}}$ is PS if and only if $\TT$ has the following partial-symmetric decomposition
\begin{equation}\label{eq:psd}
\TT = \sum_{j = 1}^{m}\lambda_j \ov{\ba_j}^{\od}  \otimes {\ba_j}^{\od},
\end{equation}
where $\lambda_j \in \C$ and $\ba_j\in \C^{n}$ for $j=1,\dots,m$ with finite $m$.
\end{corollary}
\begin{proof}
The proof of the `if' part can be straightforwardly verified by Definition~\ref{def:ps} using a PS decomposition as that in the proof of Theorem~\ref{thm:cps}.

For the `only if' part, by Proposition~\ref{thm:parts}, $\TT=\HI(\TT)+\SI(\TT)$ where $\HI(\TT)$ and $\ii \SI(\TT)$ are CPS. By Theorem~\ref{thm:cps}, both $\HI(\TT)$ and $\ii \SI(\TT)$ can be decomposed into a sum of rank-one CPS tensors as in~\eqref{eq:cpsd}, with coefficients being real numbers. Therefore, $\TT=\HI(\TT)+ (-\ii)\ii \SI(\TT)$ can be decomposed into a sum of rank-one CPS tensors as in~\eqref{eq:psd}, with coefficients being complex numbers.
\end{proof}

Corollary~\ref{thm:ps} also provides an alternative definition of a PS tensor via a {\em complex} linear combination of rank-one CPS tensors, i.e.,
$$
\PS^{n^{2d}}:=\left\{ \sum_{j = 1}^{m} \lambda_j \ov{\ba_j}^{\od} \otimes {\ba_j}^{\od}: \lambda_j\in\C,\,\ba_j\in\C^n,\, j=1,\dots,m \right\}.
$$
In terms of rank-one decompositions shown in Theorem~\ref{thm:cps} and Corollary~\ref{thm:ps}, CPS tensors and PS tensors are straightforward generalization of Hermitian matrices and complex matrices, respectively.

Some remarks on decompositions of PS tensors are necessary in place. From Definition~\ref{def:ps}, in particular the symmetricity with respect to its first half modes and symmetricity with respect to its last half modes, it can be shown that any PS tensor $\TT\in\PS^{n^{2d}}$ can also be decomposed as
\begin{equation} \label{eq:psdnew}
  \TT=\sum_{j=1}^m{\ba_j}^{\od}\otimes{\bb_j}^{\od},
\end{equation}
where $\ba_j,\,\bb_j\in\C^n$ for $j=1,\dots,m$ with some finite $m$. This decomposition seems natural from its original definition, but is quite different to and less symmetric than~\eqref{eq:psd} in Corollary~\ref{thm:ps}. In fact,~\eqref{eq:psdnew} can be immediately obtained from~\eqref{eq:psd} by absorbing each $\lambda_j$ into $\ov{\ba_j}^{\od}$. This makes the decomposition~\eqref{eq:psd} interesting as it links the first half and the last half modes of a PS tensor, which is not obvious either from Definition~\ref{def:ps} or the decomposition~\eqref{eq:psdnew}. Even for $d=1$,~\eqref{eq:psd} reduces to that any complex matrix $A\in\C^{n^2}$ can be written as $A = \sum_{j = 1}^m \lambda_j{\ba_j}{\ba_j}^{\HH}$ with $\lambda_j\in\C$ and $\ba_j\in\C^n$ for $j = 1,\dots,m$, which, to the best of our knowledge, has not been seen in the literature. However, the connection between CPS tensors and PS tensors makes~\eqref{eq:psd} more straightforward as a consequence of Theorem~\ref{thm:cps}.

\subsection{Partial-symmetric rank}

The discussion in Section~\ref{sec:psdec} obviously raises the question for the shortest partial-symmetric decomposition, a common question of interest for matrices and high-order tensors, called rank. For any tensor $\TT\in\C^{n^d}$, the rank of $\TT$, denoted by $\rank(\TT)$, is the smallest number $r$ that $\TT$ can be written as a sum of rank-one complex tensors, i.e.,
$$
\rank(\TT): = \min\left\{r :\TT = \sum_{j = 1}^r \ba_{j1}\otimes\dots\otimes\ba_{jd}, \,\ba_{jk}\in\C^n,\,j = 1,\dots,r, \, k=1,\dots,d \right\}.
$$

Depends on the types of rank-one tensors, we define the partial-symmetric rank and the conjugate partial-symmetric rank as follows.
\begin{definition}\label{def:psrank}
The partial-symmetric rank (PS rank) of a PS tensor $\TT\in\PS^{n^{2d}}$, denoted by $\rank_\ps(\TT)$, is defined as
$$
\rank_\ps(\TT): = \min\left\{r :\TT = \sum_{j = 1}^r \lambda_j \ov{\ba_j}^{\od}\otimes{\ba_j}^{\od},\,
\lambda_j\in\C,\,\ba_j\in\C^n,\,j = 1,\dots,r \right\}.
$$
The conjugate partial-symmetric rank (CPS rank) of a CPS tensor $\TT\in\CPS^{n^{2d}}$, denoted by $\rank_\cps(\TT)$, is defined as
$$
\rank_\cps(\TT): = \min\left\{r :\TT = \sum_{j = 1}^r \lambda_j \ov{\ba_j}^{\od}\otimes{\ba_j}^{\od},\,
\lambda_j\in\R,\,\ba_j\in\C^n,\,j = 1,\dots,r \right\}.
$$
\end{definition}

To echo the discussion at the end of Section~\ref{sec:psdec}, we remark that by the original definition (Definition~\ref{def:ps}) of PS tensors, another rank for PS tensors can be defined based on the decomposition~\eqref{eq:psdnew}, i.e., the minimum $r$ such that $\TT = \sum_{j = 1}^r {\ba_j}^{\od}\otimes{\bb_j}^{\od}$. This rank is different to the PS rank in Definition~\ref{def:psrank} (see Example~\ref{ex:notsame}), and is not in the scope of this paper. Our interest here is to emphasize the conjugate property and to better understand CPS tensors.

Obviously by Definition~\ref{def:psrank}, for a CPS tensor $\TT$, one has
\begin{equation}
  \rank(\TT) \le \rank_\ps(\TT) \le \rank_\cps(\TT). \label{eq:ranksame}
\end{equation}
An interesting question is whether the above inequality is an equality or not. It is obvious that~\eqref{eq:ranksame} holds at equality when the rank, PS rank, or CPS rank of a CPS tensor is one. The equality also holds in the case of matrices, i.e., for any Hermitian matrix, the three ranks must be the same. However, this is not true in general for high-order CPS tensors, as stipulated in Theorem~\ref{thm:ranknotesame}.

In $\CS^{n^d}$, the space of symmetric tensors, a similar problem was posed by Comon: The symmetric rank of a symmetric tensor is equal to the rank of the tensor, known as Comon's conjecture~\cite{CGLM08}. It received a considerable amount of attention in recent years; see e.g.,~\cite{S17} and references therein. Comon's conjecture was shown to be true in various special cases and was claimed to be invalid by a sophisticate counter example of complex symmetric tensor in a recent manuscript~\cite{S17}. Nevertheless, the real version of Comon's conjecture remains open. Our result on the ranks of CPS tensors below can be taken as a disproof for the conjugate version of Comon's conjecture. In fact, our counter example (Example~\ref{ex:notsame}) is very simple.

\begin{theorem}\label{thm:ranknotesame}
If $\TT\in\CPS^{n^{2d}}$ is a CPS tensor, then
$$\rank(\TT)\le \rank_\ps(\TT)=\rank_\cps(\TT).$$
Moreover, there exists a CPS tensor $\TT$ such that $\rank(\TT)< \rank_\ps(\TT)$.
\end{theorem}
\begin{proof}
Let $\rank_\ps(\TT)=r$ and $\TT$ has the following PS decomposition
$$\TT=\sum_{j = 1}^r \lambda_j \ov{\ba_j}^{\od}\otimes{\ba_j}^{\od}.$$
where $\lambda_j\in\C$ and $\ba_j\in\C^n$ for $j = 1,\dots,r$. It is easy to see that $\TT$ can be written as
$$
\TT= \sum_{j = 1}^r (\re\lambda_j) \ov{\ba_j}^{\od}\otimes{\ba_j}^{\od} + \ii \sum_{j = 1}^r (\im\lambda_j) \ov{\ba_j}^{\od}\otimes{\ba_j}^{\od}.
$$
We notice that both $\sum_{j = 1}^r (\re\lambda_j) \ov{\ba_j}^{\od}\otimes{\ba_j}^{\od}$ and $\sum_{j = 1}^r (\im\lambda_j) \ov{\ba_j}^{\od}\otimes{\ba_j}^{\od}$ are CPS tensors. By the uniqueness result in Proposition~\ref{thm:parts} and the fact that $\TT$ is already CPS, $\sum_{j = 1}^r (\im\lambda_j) \ov{\ba_j}^{\od}\otimes{\ba_j}^{\od}$ must be a zero tensor. Therefore,
$$
\TT= \sum_{j = 1}^r (\re\lambda_j) \ov{\ba_j}^{\od}\otimes{\ba_j}^{\od}.
$$
This implies that $\rank_\cps(\TT)\le r=\rank_\ps(\TT)$ since $\re\lambda_j\in\R$. Together with the obvious fact that $\rank_\cps(\TT)\ge\rank_\ps(\TT)$ we conclude $\rank_\cps(\TT)=\rank_\ps(\TT)$.

Example~\ref{ex:notsame} shows a CPS tensor $\TT$ with $\rank(\TT)< \rank_\ps(\TT)$.
\end{proof}

\begin{example}\label{ex:notsame}
  Let $\TT\in\CPS^{2^4}$ where $\TT_{1122}=\TT_{2211}=1$ and other entries are zeros. It follows that
  $$\rank(\TT)=2<\rank_\ps(\TT)=\rank_\cps(\TT).$$
\end{example}
\begin{proof}
Obviously $\TT$ can be written as a sum of two rank-one tensors, each matching a nonzero entry of $\TT$. It is also easy to show that $\rank(\TT)\ne1$ by contradiction. Therefore, $\rank(\TT)=2$.

We now prove $\rank_\cps(\TT)\ge3$ by contradiction. By Theorem~\ref{thm:ranknotesame}, $\rank_\cps(\TT)=\rank_\ps(\TT)\ge2$. Suppose on the contrary one has $\rank_\cps(\TT)=\rank_\ps(\TT)=2$. There exist nonzero $\lambda_1,\lambda_2\in\R$ and $\bu=(u_1,u_2)^{\T},\bv=(v_1,v_2)^{\T}\in\C^2$ such that
$$\TT=\lambda_1\ov{\bu}\otimes\ov{\bu}\otimes\bu\otimes\bu+\lambda_2\ov{\bv}\otimes\ov{\bv}\otimes\bv\otimes\bv.$$
By comparing the entries $\TT_{1111}=\TT_{1112}=\TT_{2222}=0$ and $\TT_{1122}=1$, we obtain
\begin{subequations}\label{eq:1}
\begin{align}
0&=\lambda_1|u_1|^4+\lambda_2|v_1|^4, \label{eq:1A}\\
0&=\lambda_1|u_1|^2\ov{u_1}u_2+\lambda_2|v_1|^2\ov{v_1}v_2, \label{eq:1B}\\
0&=\lambda_1|u_2|^4+\lambda_2|v_2|^4, \label{eq:1D}\\
1&=\lambda_1\ov{u_1}^2{u_2}^2+\lambda_2\ov{v_1}^2{v_2}^2. \label{eq:1E}
\end{align}
\end{subequations}

First, we claim that none of $u_1,u_2,v_1,v_2$ can be zero. Otherwise, if either $u_1$ or $v_1$ is zero, we have both $u_1$ and $v_1$ are zeros by~\eqref{eq:1A}, which invalidates~\eqref{eq:1E}. In the other case, if either $u_2$ or $v_2$ is zero, we have both $u_2$ and $v_2$ are zeros by~\eqref{eq:1D}, which also invalidates~\eqref{eq:1E}.

Let us now multiply $u_2$ to~\eqref{eq:1A} and multiply $u_1$ to~\eqref{eq:1B}, and we obtain
\begin{align*}
0&=\lambda_1|u_1|^4u_2+\lambda_2|v_1|^4u_2, \\
0&=\lambda_1|u_1|^4u_2+\lambda_2|v_1|^2\ov{v_1}v_2u_1.
\end{align*}
Combining the above two equations leads to
$$ \lambda_2|v_1|^2\ov{v_1}(v_1u_2-v_2u_1)=0,$$
which implies that $v_1u_2=v_2u_1$, i.e., $\frac{u_1}{v_1}=\frac{u_2}{v_2}$. There exists $\alpha\in\C$ such that $\bu=\alpha\bv$, and so we get
$$\TT=\lambda_1\ov{\alpha\bv}\otimes\ov{\alpha\bv}\otimes(\alpha\bv)\otimes(\alpha\bv) +\lambda_2\ov{\bv}\otimes\ov{\bv}\otimes\bv\otimes\bv = (\lambda_1|\alpha|^4+\lambda_2)\ov{\bv}\otimes\ov{\bv}\otimes\bv\otimes\bv.$$
Therefore, we arrive at $\rank_\cps(\TT)\le1$, which is obviously a contradiction.
\end{proof}

Although Example~\ref{ex:notsame} invalids the conjugate version of Comon's conjecture, the rank and the PS rank of a generic PS tensor (including CPS tensor) can still be the same when its PS rank is no more than its dimension; see Proposition~\ref{thm:ranksame}. This is similar to~\cite[Proposition 5.3]{CGLM08} for a generic symmetric complex tensor.

\begin{proposition}\label{thm:ranksame}
 If a PS tensor $\TT\in\PS^{n^{2d}}$ satisfies $\rank_\ps(\TT)\le n$, then $\rank(\TT)=\rank_\ps(\TT)$ generically.
\end{proposition}
\begin{proof}
Let $\rank(\TT)=r$ and $\rank_\ps(\TT)=m\le n$. There exist decompositions
\begin{equation}\label{eq:rankcomon}
\sum_{j=1}^r\bc_{j1}\otimes  \dots\otimes\bc_{j2d} = \TT =\sum_{j=1}^{m}\lambda_j {\ba_j}^{\od}\otimes\ov{\ba_j}^{\od},
\end{equation}
where {$\bc_{jk}\in\C^n$} for $j=1,\dots,r$ and $k=1,\dots,2d$, and nonzero $\lambda_j\in \C$ and $\ba_j\in\C^n$ for $j=1,\dots,m$. As $m\le n$, it is not difficulty to show that the set of $n$-dimensional vectors $\{\ba_1,\dots,\ba_m\}$ are generically linearly independent; see e.g.,~\cite[Lemma 5.2]{CGLM08}. As a consequence, one may find $\bx_j\in\C^n$ for $j=1,\dots,m$, such that
        $$
		\langle \bx_i, \ba_j \rangle = \left\{
		\begin{array}{ll}
			1& i=j\\
			0& i\ne j.
		\end{array}\right.
        $$
By applying the multilinear form of $\bullet\,\ov{\bx_k}^{d-1}{\bx_k}^{d}$ on both sides of~\eqref{eq:rankcomon}, we obtain
\begin{align*}
\left(\sum_{j=1}^r\bc_{j1}\otimes  \dots\otimes\bc_{j2d}\right) (\bullet\,\ov{\bx_k}^{d-1}{\bx_k}^{d}) &=  \sum_{j=1}^r \left(\prod_{i=2}^{d} \langle \ov{\bc_{ji}}, \ov{\bx_k} \rangle\right)
\left(\prod_{i=d+1}^{2d} \langle \ov{\bc_{ji}}, \bx_k \rangle\right)\bc_{j1} \\
\left(\sum_{j=1}^{m}\lambda_j {\ba_j}^{\od}\otimes \ov{\ba_j}^{\od}\right) (\bullet\,\ov{\bx_k}^{d-1}{\bx_k}^{d}) &= \lambda_k \ba_k.
\end{align*}
Therefore, for every $k=1,\dots,m$,
$$
\ba_k = \frac{1}{\lambda_k}\sum_{j=1}^r \left(\prod_{i=2}^{d} \langle \ov{\bc_{ji}}, \ov{\bx_k} \rangle\right)
\left(\prod_{i=d+1}^{2d} \langle \ov{\bc_{ji}}, \bx_k \rangle\right)\bc_{j1},
$$
i.e., a complex linear combination of $\{\bc_{11},\dots,\bc_{r1}\}$. This implies that $m\le r$. Combining with the obvious fact that $r=\rank(\TT)\le \rank_\ps(\TT)=m$, we obtain that $r=m$. In other words, $\rank(\TT)=\rank_\ps(\TT)$ holds generically.
\end{proof}
We remark that the above result already holds for CPS tensors. This is because CPS tensors are PS, and PS rank of a CPS tensor is equal to CPS rank of the tensor (Theorem~\ref{thm:ranknotesame}).

\section{Rank-one approximation and matricization equivalence}\label{sec:p2}

Finding tensor ranks are in general very hard~\cite{HL13}. This makes low-rank approximations important, and in fact it has been one of the main questions for high-order tensors. Along this line, rank-one approximation is perhaps the most simple and important topic. In this section, we study several rank-one approximations and the rank-one equivalence via matricization for CPS tensors. As an application of the matricization equivalence, new convex optimization models are developed to find the best rank-one approximation of CPS tensors.

\subsection{Rank-one approximation}

It is well known that finding the best rank-one approximation of a real tensor is equivalent to finding the largest singular value~\cite{L05} of the tensor; see e.g.,~\cite{L11}. For a real symmetric tensor, the best rank-one approximation can be obtained at a symmetric rank-one tensor~\cite{B38}, and is equivalent to finding the largest eigenvalue~\cite{Q05} of the tensor. In the complex field, Ni et al.~\cite{NQB14} studied the best symmetric rank-one approximation of symmetric complex tensors. Along this line, PS and CPS tensors possess similar properties. Let us first introduce eigenvalues of these tensors.

\begin{definition} {\em (Jiang et al.~\cite[Definition 4.4]{JLZ16}).} \label{def:eigencps}
$\lambda\in\C$ is called a C-eigenvalue of a CPS tensor $\TT\in\CPS^{n^{2d}}$ if there exists a vector $\bx\in\C^n$ called C-eigenvector, such that $\TT(\bullet\,\ov\bx^{d-1}\bx^d)=\lambda\bx$ and $\|\bx\|=1$.
\end{definition}
All the C-eigenvalues of CPS tensors are real~\cite{JLZ16}. The C-eigenvalue of PS tensors has not been defined, and we can simply adopt Definition~\ref{def:eigencps} as the definition of C-eigenvalue and C-eigenvector for PS tensors. In this paper, as long as their is no ambiguity, we call C-eigenvalue and C-eigenvector to be the eigenvalue and the eigenvector for PS tensors (including CPS tensors), respectively, and call $(\lambda,\bx)$ in Definition~\ref{def:eigencps} to be the eigenpair. We also need to clarify a coupon of terms. For a CPS tensor $\TT$, if $\rank(\TT)=1$, then $\rank_\ps(\TT)=\rank_\cps(\TT)=1$, and so the term {\em rank-one CPS tensor} has no ambiguity. However, for a PS tensor $\TT$, $\rank(\TT)=1$ does not imply $\rank_\ps(\TT)=1$. Here, the term {\em rank-one PS tensor} stands for a PS tensor $\TT$ with $\rank_\ps(\TT)=1$.

\begin{theorem} \label{thm:psrankeig}
  For a PS tensor $\TT\in\PS^{n^{2d}}$, {$\lambda\in \C$} is a largest (in terms of the modulus) eigenvalue in an eigenpair $(\lambda,\bx)$ of $\TT$ if and only if $\lambda\,\bx^{\od}\otimes\ov\bx^{\od}$ is a best rank-one PS tensor approximation of $\TT$, i.e.,
  \begin{equation}\label{eq:psrankeig}
  \argmax_{\TT(\bullet\,\ov\bx^{d-1}\bx^d)=\lambda\bx,\,\|\bx\|=1,\,\lambda\in\C}|\lambda| =\argmin_{\|\bx\|=1,\,\lambda\in\C} \|\TT-\lambda\, \bx^{\od}\otimes\ov\bx^{\od}\|.
  \end{equation}
\end{theorem}
\begin{proof}
Straightforward computation shows that
$$
\|\TT-\lambda\,\bx^{\od}\otimes\ov\bx^{\od}\|^2 = \|\TT\|^2 + |\lambda|^2 - 2\,\re(\ov\lambda \TT(\ov\bx^d\bx^d)).
$$
To minimize the right hand side of the above for given $\TT$ and fixed $\bx$, $\lambda\in\C$ must satisfy $\arg(\lambda)=\arg(\TT(\ov\bx^d\bx^d))$, which implies that $\re(\ov\lambda \TT(\ov\bx^d\bx^d))= |\lambda|\cdot|\TT(\ov\bx^d\bx^d)|$. We have
$$\min_{\lambda\in\C}\left( \|\TT\|^2 + |\lambda|^2 - 2|\lambda|\cdot|\TT(\ov\bx^d\bx^d)|\right)=\|\TT\|^2-|\TT(\ov\bx^d\bx^d)|^2,$$
held if and only if $|\lambda|=|\TT(\ov\bx^d\bx^d)|$. This further implies that $\lambda=\TT(\ov\bx^d\bx^d)$ is an optimal solution of the right hand side of~\eqref{eq:psrankeig}. Therefore, we arrive at
$$
  \argmin_{\|\bx\|=1,\lambda\in\C} \|\TT-\lambda\, \bx^{\od}\otimes\ov\bx^{\od}\| =\argmax_{\|\bx\|=1,\,\TT(\ov\bx^d\bx^d)=\lambda} |\TT(\ov\bx^d\bx^d)|
  =\argmax_{\|\bx\|=1,\,\TT(\ov\bx^d\bx^d)=\lambda} |\lambda|.
$$
By comparing to the left hand side of~\eqref{eq:psrankeig}, it suffices to show that
  \begin{equation}\label{eq:psrankeig2}
  \argmax_{\TT(\bullet\,\ov\bx^{d-1}\bx^d)=\lambda\bx,\,\|\bx\|=1,\,\lambda\in\C}|\lambda| =\argmax_{\|\bx\|=1,\,\TT(\ov\bx^d\bx^d)=\lambda} |\lambda|.
  \end{equation}
It is obvious that $\TT(\bullet\,\ov\bx^{d-1}\bx^d)=\lambda\bx$ implies $\TT(\ov\bx^d\bx^d)=\lambda$ by pre-multiplying $\ov\bx$ on both sides. It remains to prove that an optimal solution of the right hand side of~\eqref{eq:psrankeig2} satisfies $\TT(\bullet\,\ov\bx^{d-1}\bx^d)=\lambda\bx$, i.e., $\bx$ is an eigenvector of $\TT$.

The right hand side of~\eqref{eq:psrankeig2} is equivalent to $\max_{\|\bx\|^2=1}|\TT(\ov\bx^d\bx^d)|^2$. Let $\TT=\U+\ii\V$ where $\U,\V\in\CPS^{n^{2d}}$ as in Proposition~\ref{thm:parts}. This problem is further equivalent to
\begin{equation}\label{eq:extend}
\max_{\|\bx-\by\|^2 = 0,\, \|\bx\|^2 = 1,\,|\by\|^2 = 1} \left(\U(\ov\bx^d\bx^d)\right)^2 + \left(\V(\ov\by^d\by^d)\right)^2.
\end{equation}
Since both $\U(\ov\bx^d\bx^d)$ and $\V(\ov\by^d\by^d)$ are real, the Lagrangian function is
$$
f(\bx,\by,\gamma_1,\gamma_2,\gamma_3) = \left(\U(\ov\bx^d\bx^d)\right)^2 + \left(\V(\ov\by^d\by^d)\right)^2 + \gamma_1\|\bx-\by\|^2 + \gamma_2(1-\|\bx\|^2) + \gamma_3(1-\|\by\|^2).
$$
This provides (part of) the first-order optimality condition:
\begin{align*}
\frac{\partial L}{\partial \ov\bx} &= 2d\,\U(\ov\bx^d\bx^d)\, \U(\bullet\,\ov\bx^{d-1}\bx^d) + \gamma_1(\bx-\by)-\gamma_2\bx = 0, \\
\frac{\partial L}{\partial \ov\by} &= 2d\,\V(\ov\by^d\by^d)\, \V(\bullet\,\ov\by^{d-1}\by^d) - \gamma_1(\bx-\by)-\gamma_3\by = 0.
\end{align*}
By that $\bx=\by$ in the constraints of~\eqref{eq:extend}, the above equations lead to
\begin{align*}
\U(\bullet\,\ov\bx^{d-1}\bx^d) &=  \frac{\gamma_2\bx}{2d\,\U(\ov\bx^d\bx^d)},\\
\V(\bullet\,\ov\bx^{d-1}\bx^d) &=  \frac{\gamma_3\bx}{2d\,\V(\ov\bx^d\bx^d)},
\end{align*}
which implies that
$$
\TT(\bullet\,\ov\bx^{d-1}\bx^d) =  \U(\bullet\,\ov\bx^{d-1}\bx^d) + \ii \V(\bullet\,\ov\bx^{d-1}\bx^d)
= \left( \frac{\gamma_2}{2d\,\U(\ov\bx^d\bx^d)} + \frac{\ii \gamma_3}{2d\,\V(\ov\bx^d\bx^d)} \right)\bx.
$$
Therefore, $\bx$ is a an eigenvector of $\TT$.
\end{proof}

Since CPS tensors are PS tensors, for a best rank-one PS tensor $\lambda\,\bx^{\od}\otimes\ov\bx^{\od}$ approximation in~\eqref{eq:psrankeig}, $\lambda$ must be an eigenvalue. As all the eigenvalues of CPS tensors are real, $\lambda\,\bx^{\od}\otimes\ov\bx^{\od}$ becomes a best rank-one CPS tensor approximation. Therefore, Theorem~\ref{thm:psrankeig} immediately implies a similar result for CPS tensors.
\begin{corollary} \label{thm:cpsrankeig}
For a CPS tensor $\TT\in\CPS^{n^{2d}}$, {$\lambda\in\R$} is a largest (in terms of the absolute value) eigenvalue in an eigenpair $(\lambda,\bx)$ of $\TT$ if and only if $\lambda\,\bx^{\od}\otimes\ov\bx^{\od}$ is a best rank-one CPS tensor approximation of $\TT$, i.e.,
  $$
  \argmax_{\TT(\bullet\,\ov\bx^{d-1}\bx^d)=\lambda\bx,\,\|\bx\|=1,\,\lambda\in\R}|\lambda| =\argmin_{\|\bx\|=1,\,\lambda\in\R} \|\TT-\lambda\, \bx^{\od}\otimes\ov\bx^{\od}\|.
  $$
\end{corollary}

For CPS tensors, one may consider different rank-one approximation problems.
The following result is interesting, which echoes the inequivalence on ranks discussed earlier in Theorem~\ref{thm:ranknotesame}.
\begin{theorem} \label{thm:rankoneapp}
  If $\TT\in\CPS^{n^{2d}}$ is a CPS tensor with $d\ge 2$, then the best rank-one CPS tensor approximation of $\TT$ is equivalent to the best rank-one PS tensor approximation of $\TT$, but is not equivalent to the best rank-one complex tensor approximation of $\TT$, i.e.,
  \begin{equation} \label{eq:rankoneapp}
  \min_{\rank_\cps(\X)=1,\,\X\in\CPS^{n^{2d}}}\|\TT-\X\| = \min_{\rank_\ps(\X)=1,\,\X\in\PS^{n^{2d}}}\|\TT-\X\| \ge \min_{\rank(\X)=1,\,\X\in\C^{n^{2d}}}\|\TT-\X\|.
  \end{equation}
  Moreover, there exists a CPS tensor $\TT$ such that
  $$\min_{\rank_\cps(\X)=1,\,\X\in\CPS^{n^{2d}}}\|\TT-\X\| > \min_{\rank(\X)=1,\,\X\in\C^{n^{2d}}}\|\TT-\X\|.$$
\end{theorem}

In fact, the equality in~\eqref{eq:rankoneapp} is a consequence of Theorem~\ref{thm:psrankeig} and Corollary~\ref{thm:cpsrankeig}. The inequality in~\eqref{eq:rankoneapp} is obvious since $\CPS^{n^{2d}}\subset\C^{n^{2d}}$ and $\rank_\cps(\X)=1$ implies that $\rank(\X)=1$, and its strictness can be validated by the following example.
\begin{example}\label{ex:notsame2}
  Let $\TT\in\CPS^{2^4}$ where $\TT_{1122}=\TT_{2211}=1$ and other entries are zeros. For any $\bz\in\C^2$ with $\|\bz\|=1$, one has
  $$|\TT(\ov\bz^2\bz^2)|=|{\ov{z_1}}^2{z_2}^2+{\ov{z_2}}^2{z_1}^2| \le 2|z_1|^2|z_2|^2\le\frac{1}{2}(|z_1|^2+|z_2|^2)^2=\frac{1}{2},$$
  implying that
  $$
  \|\TT-\lambda\bz^{\otimes 2}\otimes \ov\bz^{\otimes 2}\|^2 =\|\TT\|^2-2\lambda\TT(\ov\bz^2\bz^2)+\lambda^2\ge \|\TT\|^2 - |\TT(\ov\bz^2\bz^2)|^2 \ge2-\frac{1}{4}=\frac{7}{4}.
  $$
  However,
  $$
  \|\TT-\be_1\otimes\be_1\otimes\be_2\otimes\be_2\|^2 =1.
  $$
  This shows that $\min_{\rank_\cps(\X)=1,\,\X\in\CPS^{2^4}}\|\TT-\X\| > \min_{\rank(\X)=1,\,\X\in\C^{2^4}}\|\TT-\X\|$.
\end{example}

We remark that the equivalence between the best rank-one CPS tensor approximation and the best rank-one complex tensor approximation actually holds for Hermitian matrices ($d=1$), i.e.,
$$
  \min_{\rank_\cps(X)=1,\,X\in\CPS^{n^2}}\|A-X\| = \min_{\rank_\ps(X)=1,\,X\in\PS^{n^2}}\|A-X\| = \min_{\rank(X)=1,\,X\in\C^{n^2}}\|A-X\|
$$
when $A\in\C^{n^2}$ is Hermitian. This is because of the trivial fact $\C^{n^2}=\PS^{n^2}$ and the equality in~\eqref{eq:rankoneapp}.

\subsection{Rank-one equivalence via matricization}

Matricization, or matrix unfolding, of a tensor is a widely used tool to study high-order tensors. When a tensor is rank-one, it is obvious that any matricization of the tensor is rank-one, while the reverse is not true in general. For an even-order symmetric tensor $\TT\in\CS^{n^{2d}}$, it is known that if its square matricization (unfolding $\TT$ as an $n^d\times n^d$ matrix) is rank-one, then the original tensor $\TT$ must be rank-one; see e.g.,~\cite{NW14, JMZ15}. In the real field, this rank-one equivalence suggests some convex optimization methods to compute the largest eigenvalue or best rank-one approximation of a symmetric tensor. In practice, the methods are very likely to find global optimal solutions~\cite{NW14,JMZ15}. Inspired by these results, let us look into the rank-one equivalence for CPS tensors.

For a CPS tensor, one hopes that its square matricization being rank-one implies the original tensor being rank-one. Unfortunately, this may not hold ture if a CPS tensor is not unfolded in a {\em right way}.
The following example shows that the standard square matricization of a non-rank-one CPS tensor turns to a rank-one Hermitian matrix.
\begin{example}\label{ex:nonrankone}
Let $\TT=\ov{A}\otimes A\in\C^{2^4}$ where
$A=\left(\begin{smallmatrix} 1 & 1 + \ii \\ 1+\ii & 2\end{smallmatrix}\right)\in\C^{2^2}$. Explicitly, $\TT$ can be written as
$$
\begin{array}{c|c}
\TT(\be_1\be_1\bullet\bullet) & \TT(\be_1\be_2\bullet\bullet) \\
\hline
\TT(\be_2\be_1\bullet\bullet) & \TT(\be_2\be_2\bullet\bullet)
\end{array}
=
\begin{array}{c|c}
\left(\begin{smallmatrix} 1 & 1+\ii \\ 1+\ii & 2\end{smallmatrix}\right) & \left(\begin{smallmatrix}1-\ii & 2 \\ 2 & 2-2\ii\end{smallmatrix}\right)  \\
\hline
\left(\begin{smallmatrix}1-\ii & 2 \\ 2 & 2-2\ii\end{smallmatrix}\right) & \left(\begin{smallmatrix} 2 & 2+2\ii \\ 2+2\ii & 4\end{smallmatrix}\right)
\end{array},
$$
which can be straightforwardly verified as a CPS tensor. However, $\rank(\TT)\ge 2$ but the standard square matricization of $\TT$ is a rank-one Hermitian matrix.
\end{example}
\begin{proof}
By the construction of $\TT$ via the outer product of two matrices, it is obvious that its standard square matricization is $(1,1+\ii,1+\ii,2)^{\HH}(1,1+\ii,1+\ii,2)$, which is a rank-one Hermitian matrix.

On the other hand, suppose that $\rank(\TT)=1$. This implies that $\rank_\cps(\TT)=1$ and so we may let $\TT=\ov{\bx}\otimes\ov{\bx}\otimes\bx\otimes\bx$ for some $\bx\in\C^2$. By comparing some entries, one has
$$
|x_1|^4=\TT_{1111}=1, \,|x_2|^4=\TT_{2222}=4, \, \ov{x_1}^2x_1x_2=\TT_{1112}=1+\ii, \mbox{ and } \ov{x_2}\ov{x_1}x_2x_2=\TT_{2122}=2-2\ii.
$$
Clearly $|x_1|^2=1$ and $|x_2|^2=2$, and this leads to
$$
2-2\ii = \ov{x_2}\ov{x_1}x_2x_2 = 2\,\ov{x_1}x_2 =  2\, \ov{x_1}^2x_1x_2 = 2 + 2\ii,
$$
a contradiction. Therefore, $\rank(\TT)\ge2$.
\end{proof}

We notice that square matricization is unique for symmetric tensors, but not for CPS tensors. Example~\ref{ex:nonrankone} motivates us to consider other ways of matricization, with a hope to establish certain rank-one equivalence. To this end, it is necessary to introduce tensor transpose, extending the concept of matrix transpose.
%
\begin{definition}\label{def:tensorp}
Given a tensor $\TT\in\C^{n^d}$ and a permutation $\pi=(\pi_1,\dots,\pi_d)\in\Pi(1,\dots,d)$, the $\pi$-transpose of $\TT$, denoted by $\TT^\pi\in\C^{n^d}$, satisfies
$$\TT_{i_1\dots i_d}=(\TT^\pi)_{i_{\pi_1} \dots i_{\pi_{d}}}  \quad\forall\, 1 \le i_1, \dots,  i_d\le n.$$
\end{definition}
In a plain language, mode $1$ of $\TT^\pi$ originates from mode $\pi_1$ of $\TT$, mode $2$ of $\TT^\pi$ originates from mode $\pi_2$ of $\TT$, and so on. As a matter of fact, for a matrix $A\in\C^{n^2}$ and $\pi=(2,1)$, $A^\pi=A^{\T}$. For a PS tensor $\TT\in\PS^{n^{2d}}$ and $\pi=(d+1,\dots,2d,1,\dots,d)$, $\TT^{\HH}=\ov{\TT^\pi}$.

Given any integers $1\le i_1,\dots,i_d\le n$, let us denote $n(i_1\dots i_d):=\sum_{k = 1}^{d}(i_k-1)n^{d-k} + 1$
to be the decimal of the tuple $i_1\dots i_d$ in the base-$n$ numeral system.
We now discuss matricization and vectorization.
\begin{definition}
Given a tensor $\TT\in\C^{n^{d}}$, the vectorization of $\TT$, denoted by $\bv(\TT)$, is an $n^{d}$-dimensional vector satisfying
$$v(\TT)_{n(i_1\dots i_d)} = \TT_{i_1\dots i_{d}} \quad \forall\, 1\le i_1,\dots ,i_{d}\le n.$$
Given an even-order tensor $\TT\in\C^{n^{2d}}$, the standard square matricization (or simply matricization) of $\TT$, denoted by $M(\TT)$, is an $n^d\times n^d$ matrix satisfying
$$M(\TT)_{n(i_1\dots i_d)\, n(i_{d+1}\dots i_{2d})} = \TT_{i_1\dots i_{2d}} \quad \forall\, 1\le i_1,\dots ,i_{2d}\le n.$$
\end{definition}

\begin{definition}\label{def:matrix}
Given an even-order tensor $\TT\in\C^{n^{2d}}$ and a permutation $\pi\in\Pi(1,\dots,2d)$, the $\pi$-matricization of $\TT$, denoted by $M_\pi(\TT)$, satisfies
$$M_\pi(\TT)_{n(i_{\pi_1}\dots i_{\pi_d})\, n(i_{\pi_{d+1}}\dots i_{\pi_{2d}})}= \TT_{i_{1} \dots i_{{2d}}} \quad\forall\, 1 \le i_1, \dots,  i_{2d}\le n,$$
in other words, $M_\pi(\TT)=M(\TT^\pi)$.
\end{definition}

Obviously the standard square matricization is a $\pi$-matricization when $\pi=(1,2,\dots,2d)$. Vectorization, $\pi$-matricization and $\pi$-transpose are all one-to-one. They are different ways of representation for tensor data. The following property on ranks are straightforward.
\begin{proposition}\label{thm:pi}
Given a tensor $\TT\in\C^{n^{2d}}$ and any $\pi\in\Pi(1,\dots,2d)$, it follows that $\rank(\TT)=\rank(\TT^\pi)$ and $\rank(M_\pi(\TT))\le \rank(\TT)$, in particular, $\rank(\TT)=1 \Longrightarrow \rank(M_\pi(\TT))=1$.
\end{proposition}
As mentioned earlier, the $\pi$-matricization of a symmetric tensor is unique for any permutation $\pi$, since $\TT^\pi=\TT$ if $\TT$ is symmetric. However, CPS tensors only possess partial symmetricity as well as certain conjugate property. Therefore, conditions of $\pi$ are necessary to guarantee the rank-one equivalence, as well as for the $\pi$-matricization being Hermitian.

\begin{proposition}\label{thm:conjeq}
If $\TT\in\CPS^{n^{2d}}$ and a permutation $\pi\in\Pi(1,\dots,2d)$ satisfies
 \begin{equation} \label{eq:conjeq}
 \left|\{\pi_k,\pi_{d+k}\}\cap \{1,\dots,d\}\right|=1 \quad\forall\, k=1,\dots,d,
 \end{equation}
then $M_\pi(\TT)$ is a CPS (Hermitian) matrix.
\end{proposition}
\begin{proof}
Since any CPS tensor can be written as a sum of rank-one CPS tensors (Theorem~\ref{thm:cps}), we only need to show the case $\TT$ is rank-one. Suppose that $\TT=\ba_1\otimes\dots\otimes \ba_{2d}$, where $\ba_1=\dots=\ba_d=\ov\bx$ and $\ba_{d+1}=\dots=\ba_{2d}=\bx$, and so $\TT^\pi=\ba_{\pi_1}\otimes\dots\otimes \ba_{\pi_{2d}}$.

For any $k\in\{1,\dots,d\}$, as exactly one of $\{\pi_k,\pi_{d+k}\}$ belongs to $\{1,\dots,d\}$ and the other belongs to $\{d+1,\dots,2d\}$, we have $\ov{\ba_{\pi_k}}=\ba_{\pi_{d+k}}$. Therefore,
\begin{align*}
 M(\TT^\pi)^{\HH}& = \left(\bv(\ba_{\pi_1}\otimes\dots\otimes \ba_{\pi_d}) \otimes \bv(\ba_{\pi_{d+1}}\otimes\dots\otimes \ba_{\pi_{2d}}) \right)^{\HH}\\
  & = \ov{\bv(\ba_{\pi_{d+1}}\otimes\dots\otimes \ba_{\pi_{2d}}) \otimes \bv(\ba_{\pi_1}\otimes\dots\otimes \ba_{\pi_d})} \\
  & = \bv(\ov{\ba_{\pi_{d+1}}}\otimes\dots\otimes \ov{\ba_{\pi_{2d}}}) \otimes \bv(\ov{\ba_{\pi_1}}\otimes\dots\otimes \ov{\ba_{\pi_d}}) \\
  & = \bv(\ba_{\pi_{1}}\otimes\dots\otimes \ba_{\pi_{d}}) \otimes \bv(\ba_{\pi_{d+1}}\otimes\dots\otimes \ba_{\pi_{2d}}) \\
  & = M(\TT^\pi),
\end{align*}
proving that $M_\pi(\TT)=M(\TT^\pi)$ is a Hermitian matrix.
\end{proof}

We remark that the condition of $\pi$ in~\eqref{eq:conjeq} is in fact necessary for $M_\pi(\TT)$ to be a Hermitian matrix for a general CPS tensor $\TT$. Essentially, if mode $k$ of $\TT^\pi$ originates from modes $\{1,\dots,d\}$ of $\TT$, then mode $d+k$ of $\TT^\pi$ must originate from modes $\{d+1,\dots,2d\}$ of $\TT$, and vise versa.

\begin{theorem}\label{thm:rankeq}
If $\TT\in\PS^{n^{2d}}$ and $\pi\in\Pi(1,\dots,2d)$ satisfies
\begin{equation}\label{eq:rankeq}
 \left\lfloor\frac{d}{2}\right\rfloor \le \left|\{\pi_1,\dots,\pi_{d}\}\cap \{1,\dots,d\}\right| \le \left\lceil\frac{d}{2}\right\rceil,
\end{equation}
then $\rank(M_\pi(\TT))=1 \Longrightarrow \rank(\TT)=1$.
\end{theorem}
\begin{proof}
Let $M_\pi(\TT)=\bx\otimes\by$ where $\bx$ and $\by$ are $n^d$-dimensional vectors, and further let $\bx=\bv(\X)$ and $\by=\bv(\Y)$ where $\X,\Y\in\C^{n^d}$. Since
$$M(\TT^\pi)=M_\pi(\TT)=\bx\otimes\by=\bv(\X)\otimes\bv(\Y),$$
one has $\TT^\pi=\X\otimes\Y$. To prove $\rank(\TT)=1$, it suffices to show that $\rank(\X)=\rank(\Y)=1$.

The modes $\{1,\dots,d\}$ of $\TT^\pi$ are originated from modes $\{\pi_1,\dots,\pi_d\}$ of $\TT$. By~\eqref{eq:rankeq}, there are almost half (either $\left\lfloor d/2\right\rfloor$ or $\left\lceil d/2 \right\rceil$) from modes $\{1,\dots,d\}$ of $\TT$ and the remaining half from modes $\{d+1,\dots,2d\}$ of $\TT$. This is also true for modes of $\X$. To provide a clearer presentation, we may construct a permutation $\rho$ such that, the first half modes of $\X^\rho$ are from modes $\{1,\dots,d\}$ of $\TT$ and the remaining modes of $\X^\rho$ are from modes $\{d+1,\dots,2d\}$ of $\TT$. Explicitly, $\rho\in\Pi(1,\dots,d)$ needs to satisfy
\begin{equation}\label{eq:rho}
\left\{ \rho_1,\dots,\rho_{\left|\{\pi_1,\dots,\pi_{d}\}\cap \{1,\dots,d\}\right|}\right\} =
\left\{1\le k\le d: 1\le \pi_k\le d \right\}.
\end{equation}
As $\TT$ is symmetric to modes $\{1,\dots,d\}$ and to modes $\{d+1,\dots,2d\}$, respectively, the order of $\rho_k$'s in~\eqref{eq:rho} does not matter. Observing that $\rank(\X)=\rank(\X^\rho)$ and $M(\X^\rho\otimes\Y)$ is rank-one, we may without loss of generality assume that the first $\left\lceil d/2 \right\rceil$ modes of $\X$ are from modes $\{1,\dots,d\}$ of $\TT$ and the remaining $\left\lfloor d/2 \right\rfloor$ from modes $\{d+1,\dots,2d\}$ of $\TT$, and for the same reason assume that the first $\left\lfloor  d/2 \right\rfloor$ modes of $\Y$ are from modes $\{1,\dots,d\}$ of $\TT$ and the remaining $\left\lceil d/2 \right\rceil$ from modes $\{d+1,\dots,2d\}$ of $\TT$. In a nutshell, we assume without loss of generality that
$$
\pi=\left(1,\dots, \left\lceil\frac{d}{2}\right\rceil, d+1,\dots, d+ \left\lfloor\frac{d}{2}\right\rfloor, \left\lceil\frac{d}{2}\right\rceil+1, \dots, d, d+\left\lfloor\frac{d}{2}\right\rfloor+1,\dots, 2d \right).
$$

We proceed to prove $\rank(\X)=\rank(\Y)=1$ by induction on $d$, as long as $\X\otimes \Y=\TT^\pi$ is symmetric to modes
$$\I_d^1:=\left\{1,\dots,\left\lceil\frac{d}{2}\right\rceil,d+1,\dots,d+\left\lfloor\frac{d}{2}\right\rfloor\right\},$$
and to modes
$$\I_d^2:=\left\{\left\lceil\frac{d}{2}\right\rceil+1,\dots,d,d+\left\lfloor\frac{d}{2}\right\rfloor+1,\dots,2d\right\},$$
respectively.

When $d=1$, both $\X$ and $\Y$ are obviously rank-one as they are vectors. Suppose that the claim holds for $d-1$. For general $d$, we can swap all but the first mode of $\X$ with some modes of $\Y$. In particularly, modes $\left\{2,\dots, \left\lceil\frac{d}{2}\right\rceil\right\}$ of $\X$ are swapped with modes $\left\{1,\dots, \left\lceil\frac{d}{2}\right\rceil-1\right\}$ of $\Y$, respectively, and modes $\left\{\left\lceil\frac{d}{2}\right\rceil+1, \dots,d\right\}$ of $\X$ are swapped with modes $\left\{\left\lceil\frac{d}{2}\right\rceil+1,\dots,d\right\}$ of $\Y$, respectively. Consequently, one has for any $1 \le i_1, \dots,  i_{2d}\le n$,
\begin{align}\label{eq:transfer}
  \X_{i_1\dots i_d}\Y_{i_{d+1}\dots i_{2d}} =
  \X_{ i_1 i_{d+1} \dots i_{d+\lceil d/2\rceil-1} i_{d+\lceil d/2\rceil +1} \dots i_{2d} }
  \Y_{  i_2 \dots i_{\lceil d/2\rceil} i_{d+\lceil d/2\rceil} i_{\lceil d/2\rceil+1}\dots i_d }.
\end{align}
Pick any nonzero entry of $\Y$, say $\Y_{k_1\dots k_d}\neq0$. Let $(i_{d+1}, \dots ,i_{2d})=(k_1,\dots,k_d)$ in~\eqref{eq:transfer} and we have
\begin{align*}
\X_{i_1\dots i_d}\Y_{k_1\dots k_d} =
  \X_{ i_1 k_1 \dots k_{\lceil d/2\rceil-1} k_{\lceil d/2\rceil +1} \dots k_{d} }
  \Y_{  i_2 \dots i_{\lceil d/2\rceil} k_{\lceil d/2\rceil} i_{\lceil d/2\rceil+1}\dots i_d  }.
\end{align*}
By defining $\ba\in\C^n$ and $\U\in\C^{n^{d-1}}$ where
$$
a_{i_1}:=\frac{\X_{ i_1 k_1 \dots k_{\lceil d/2\rceil-1} k_{\lceil d/2\rceil +1} \dots k_{d} }}{\Y_{k_1\dots k_d}}
 \mbox{ and } \U_{i_2\dots i_d}:=\Y_{ i_2 \dots i_{\lceil d/2\rceil} k_{\lceil d/2\rceil} i_{\lceil d/2\rceil+1}\dots i_d },
$$
we obtain that $\X=\ba\otimes \U$. Similarly to~\eqref{eq:transfer}, one may swap all but the last modes of $\Y$ to some modes of $\X$ and obtain $\Y=\V\otimes \bb$ where $\V\in\C^{n^{d-1}}$ and $\bb\in\C^n$. Since $\X\otimes \Y = \ba\otimes \U \otimes \V\otimes \bb$, we observe that $\U\otimes \V$ is symmetric to modes $\left\{1,\dots,\lceil\frac{d}{2}\rceil-1,d,\dots,d+\lfloor\frac{d}{2}\rfloor-1\right\}$ and to modes $\left\{\lceil\frac{d}{2}\rceil,\dots,d-1,d+\lfloor\frac{d}{2}\rfloor,\dots,2d-2\right\}$, respectively.

If $d$ is odd, noticing that $\lceil\frac{d}{2}\rceil-1=\lceil\frac{d-1}{2}\rceil$ and $\lfloor\frac{d}{2}\rfloor=\lfloor\frac{d-1}{2}\rfloor$, we have that $\U\otimes\V$ is symmetric to modes $\I^1_{d-1}$ and to modes $\I^2_{d-1}$, respectively. By induction, we obtain $\rank(\U)=\rank(\V)=1$, proving that $\rank(\X)=\rank(\Y)=1$.

If $d$ is even, we need to consider $\V\otimes \U$ (instead of $\U\otimes\V$), which
is symmetric to modes $\left\{1,\dots,\lfloor\frac{d}{2}\rfloor,d,\dots,d+\lceil\frac{d}{2}\rceil-2\right\}$ and to modes $\left\{\lfloor\frac{d}{2}\rfloor+1,\dots,d-1,d+\lceil\frac{d}{2}\rceil-1,\dots,2d-2\right\}$, respectively. Noticing
$\lfloor\frac{d}{2}\rfloor=\lceil\frac{d-1}{2}\rceil$ and $\lceil\frac{d}{2}\rceil=\lfloor\frac{d-1}{2}\rfloor+1$, the two sets of modes are exactly $\I^1_{d-1}$ and  $\I^2_{d-1}$, respectively. By induction, we obtain $\rank(\V)=\rank(\U)=1$, and so $\rank(\X)=\rank(\Y)=1$.
\end{proof}

In fact, the condition of $\pi$ in~\eqref{eq:rankeq} is also a necessary condition for the rank-one equivalence in Theorem~\ref{thm:rankeq} for a general CPS tensor. The proof, or an explanation of a counter example, involves heavy notations and we leave it to interested readers. The key step leading to Theorem~\ref{thm:rankeq} is the identity~\eqref{eq:transfer}, which is a consequence of modes swapping due to some partial symmetricity. This is doable because, among modes $\{1,\dots,d\}$ of $\TT$, they are (almost) equally allocated to modes of $\X$ (the first half modes of $\TT^\pi$) and to modes of $\Y$ (the last half modes of $\TT^\pi$), and the same holds for modes $\{d+1,\dots,2d\}$ of $\TT$. If the number of modes of $\X$ that originate from modes $\{1,\dots,d\}$ of $\TT$ differs the number of modes of $\Y$ that originate from modes $\{1,\dots,d\}$ of $\TT$ for more than one (such as Example~\ref{ex:nonrankone} with $\pi=(1,2,3,4)$), then~\eqref{eq:transfer} cannot be obtained. This makes some modes binding, i.e., not separable to the outer product of a vector and a tensor in a lower order.

Combing Propositions~\ref{thm:pi} and~\ref{thm:conjeq}, Theorem~\ref{thm:rankeq}, and the discussion regarding the necessities, we arrive at the following result.
\begin{theorem}\label{thm:combine}
Both $M_\pi(\TT)$ is Hermitian and $\rank(M_\pi(\TT))=1 \Longleftrightarrow \rank(\TT)=1$ hold for any CPS tensor $\TT\in\CPS^{n^{2d}}$ if and only if $\pi\in\Pi(1,\dots,2d)$ satisfies both~\eqref{eq:conjeq} and~\eqref{eq:rankeq}.
\end{theorem}

In practice, such as the discussion in modelling (Section~\ref{sec:method}) and the numerical experiments (Section~\ref{sec:numerical}), we may focus on a particular permutation satisfying~\eqref{eq:conjeq} and~\eqref{eq:rankeq}. The most straightforward one is
\begin{equation} \label{eq:bestpi}
  \pi=\left(1,\dots,\left\lceil\frac{d}{2}\right\rceil, d+1,\dots,d+\left\lfloor\frac{d}{2}\right\rfloor, d+\left\lfloor\frac{d}{2}\right\rfloor+1,\dots,2d, \left\lceil\frac{d}{2}\right\rceil+1,\dots,d \right).
\end{equation}
In particular, $\pi=(1,3,4,2)$ for fourth-order tensors ($d=2$), $\pi=(1,2,4,5,6,3)$ for sixth-order tensors ($d=3$), and $\pi=(1,2,5,6,7,8,3,4)$ for eighth-order tensors ($d=4$).

\subsection{Finding best rank-one approximation}\label{sec:method}

As an immediate application of the rank-one equivalence, we now discuss how it can be used to find the best rank-one approximation of CPS tensors. Specifically, we consider the problem
$$\min_{\|\bx\|=1,\lambda\in\R} \|\TT-\lambda\, \bx^{\od}\otimes \ov\bx^{\od}\|.$$
As mentioned in Corollary~\ref{thm:cpsrankeig},
$$\min_{\|\bx\|=1,\lambda\in\R} \|\TT-\lambda\, \bx^{\od}\otimes \ov\bx^{\od} \| \Longleftrightarrow
\max_{ \TT(\bullet\,\ov\bx^{d-1}\bx^d)=\lambda\bx,\,\|\bx\|=1,\,\lambda\in\R}|\lambda|
\Longleftrightarrow \max_{\|\bx\|=1}|\TT(\ov\bx^d\bx^d)|.$$
Since $\TT(\ov\bx^d\bx^d)$ is a real-valued function, the maximum of $|\TT(\ov\bx^d\bx^d)|$ is obtained either at $\TT(\ov\bx^d\bx^d)$ or at $(-\TT)(\ov\bx^d\bx^d)$ for a given $\TT$. Therefore, the above problem is essentially
\begin{equation} \label{eq:modeltensor}
  \max_{\|\bx\|=1}\TT(\ov\bx^d\bx^d),
\end{equation}
i.e., finding the largest eigenvalue of the CPS tensor $\TT$.


The model~\eqref{eq:modeltensor} is NP-hard when the order of $\TT$ is larger than two, even in the real field~\cite{HLZ10,HL13}. Let us now transfer the tensor based optimization model to a matrix optimization model.
\begin{theorem} \label{thm:equalmodel}
  If $\TT\in\CPS^{n^{2d}}$ and $\pi$ satisfies~\eqref{eq:conjeq} and~\eqref{eq:rankeq}, then~\eqref{eq:modeltensor} is equivalent to
\begin{equation} \label{eq:modelmatrix}
  \max\left\{\langle \ov{M_\pi(\TT)}, X \rangle: \tr(X)=1,\, \rank(X)=1,\, X\in M_\pi(\CPS^{n^{2d}}), \, X^{\HH}=X  \right\},
\end{equation}
where $M_\pi(\CPS^{n^{2d}}):=\left\{M_\pi(\X): \X\in\CPS^{n^{2d}} \right\}$.
\end{theorem}
\begin{proof}
The equivalence between~\eqref{eq:modelmatrix} and~\eqref{eq:modeltensor} can be established via $X=M_\pi(\ov\bx^{\od}\otimes \bx^{\od})$, where $X$ and $\bx$ are feasible solutions of~\eqref{eq:modelmatrix} and~\eqref{eq:modeltensor}, respectively.

Given an optimal solution $\bz$ of~\eqref{eq:modeltensor}, by Propositions~\ref{thm:pi} and~\ref{thm:conjeq}, $Z=M_\pi(\ov\bz^{\od}\otimes \bz^{\od})$ is a rank-one Hermitian matrix, and so $Z=\ov\by\otimes\by$ where $\by$ is an $n^d$-dimensional vector. Moreover,
\begin{equation} \label{eq:traceone}
\tr(Z)=\langle I, \ov\by\otimes\by \rangle = \|\by\|^2=\|Z\|=\|\ov\bz^{\od}\otimes \bz^{\od}\|=\|\bz\|^{2d}=1.
\end{equation}
This shows that $Z$ is a feasible solution of~\eqref{eq:modelmatrix}, whose objective value is
$$
\langle \ov{M_\pi(\TT)}, Z\rangle = \langle M_\pi(\ov{\TT}), M_\pi(\ov\bz^{\od}\otimes \bz^{\od})\rangle
= \langle \ov\TT, \ov\bz^{\od}\otimes \bz^{\od}\rangle = \TT(\ov\bz^d\bz^d).
$$

On the other hand, given an optimal solution $Z$ of~\eqref{eq:modelmatrix}, let $\Z\in\CPS^{n^{2d}}$ such that $Z=M_\pi(\Z)$. As $Z$ is a rank-one Hermitian matrix, $Z=\alpha\ov\by\otimes\by$ for some $\alpha\in\R$ and $\|\by\|=1$. Further by $\tr(Z)=1$ and~\eqref{eq:traceone}, we observe that $\alpha=1$ and so $Z=\ov\by\otimes\by$. Moreover, by Theorem~\ref{thm:rankeq}, $\Z$ is a rank-one CPS tensor, i.e., $\Z=\lambda \,\ov\bz^{\od}\otimes \bz^{\od}$ for some $\lambda\in\R$ and $\|\bz\|=1$. Noticing that
$$
\ov\by\otimes\by = Z = M_\pi(\Z) = M_\pi(\lambda \,\ov\bz^{\od}\otimes \bz^{\od}) \mbox{ and } \|\by\|=\|\bz\|=1,
$$
it is easy to see that $\lambda=1$, resulting $Z=M_\pi(\ov\bz^{\od}\otimes \bz^{\od})$. Therefore, $\bz$ is a feasible solution of~\eqref{eq:modeltensor}, whose objective value is
$$
\TT(\ov\bz^d\bz^d) = \langle \ov\TT, \ov\bz^{\od}\otimes \bz^{\od}\rangle =
\langle M_\pi(\ov{\TT}), M_\pi(\ov\bz^{\od}\otimes \bz^{\od})\rangle = \langle \ov{M_\pi(\TT)}, Z\rangle.
$$
\end{proof}


We remark that both $\tr(X)=1$ and $X\in M_\pi(\CPS^{n^{2d}})$ in the model~\eqref{eq:modelmatrix} are linear equality constraints. In particular, $X\in M_\pi(\CPS^{n^{2d}})$ contains $O(n^d)$ equalities, which are the requirements of partial symmetricity and conjugate property for CPS tensors. As an example, when $n=d=2$ and let $\pi=(1,3,4,2)$ as in~\eqref{eq:bestpi}, $X\in M_\pi(\CPS^{2^{4}})$ can be explicitly written as
$$
X_{14}=X_{22}= X_{33}=X_{41},\, X_{12}=X_{31},\, X_{24}=X_{43}, \mbox{ and } X^{\HH}=X.
$$
In fact, $X^{\HH}=X$ is included in the constraints $X\in M_\pi(\CPS^{n^{2d}})$, but we leave it in~\eqref{eq:modelmatrix} to emphasize that the decision variable sits in the space of Hermitian matrices.

The problem~\eqref{eq:modelmatrix} remains hard because of the rank-one constraint. However, it broadens ways by resorting to various matrix optimization tools, particularly in convex optimization. We now propose two convex relaxation methods. First, in the proof of Theorem~\ref{thm:equalmodel}, we observe that $X$ is a rank-one Hermitian matrix with $\tr(X)=1$ actually implies that $X$ is positive semidefinite. By dropping the rank-one constraint,~\eqref{eq:modelmatrix} is relaxed to a semidefinite program (SDP):
\begin{equation}\label{eq:sdp}
  \max\left\{\langle \ov{M_\pi(\TT)}, X \rangle: \tr(X)=1,\, X\in M_\pi(\CPS^{n^{2d}}), \, X\succeq O  \right\},
\end{equation}
where $X\succeq O$ denotes that $X$ is Hermtian positive semidefinite. The convex optimization model~\eqref{eq:sdp} can be easily solved by some SDP solvers in CVX~\cite{GB17}. Alternatively, one may resort to first order methods such as the alternating direction method of multipliers (ADMM).

The second relaxation method is to add a penalty of the nuclear norm of the decision matrix in the objective function~\cite{JMZ15}. By dropping the rank-one constraint, this leads to the following convex optimization model
\begin{equation}\label{eq:nuclear}
  \max\left\{\langle \ov{M_\pi(\TT)}, X \rangle-\rho\|X\|_*: \tr(X)=1,\, X\in M_\pi(\CPS^{n^{2d}}), \, X^{\HH}=X  \right\},
\end{equation}
where $\rho>0$ is a penalty parameter and $\|X\|_*$ denotes the nuclear norm of $X$, a convex surrogate for $\rank(X)$. To see why~\eqref{eq:nuclear} is a convex relaxation of~\eqref{eq:modelmatrix}, we notice that $\|X\|_*$ is a convex function, and so the objective of~\eqref{eq:nuclear} is concave. Moreover, an optimal solution of~\eqref{eq:modelmatrix}, say $X$, is rank-one and $\tr(X) = 1$ imply that $X$ is positive semidefinite. Thus, $\|X\|_* = \tr(X) = 1$, which implies that the term $-\rho\|X\|_*$ added to the objective function is actually a constant.

Our observations in several numerical examples show that the solution obtained by the two convex relaxation models~\eqref{eq:sdp} and~\eqref{eq:nuclear} are often rank-one (see Section~\ref{sec:numerical}). Once a rank-one solution $X$ is obtained, one may resort $X=M_\pi(\ov\bx^{\od}\otimes \bx^{\od})$ to find a solution $\bx$ for~\eqref{eq:modeltensor}, as stipulated in the proof of Theorem~\ref{thm:equalmodel}.

\section{Numerical experiments}\label{sec:numerical}

In this section, we conduct numerical experiments to test the methods proposed in Section~\ref{sec:method} in finding the best rank-one approximation of CPS tensors. This is also to justify applicability of the rank-one equivalence in Theorem~\ref{thm:rankeq} or Theorem~\ref{thm:combine}. Both the nuclear norm penalty model~\eqref{eq:nuclear} and the SDP relaxation method~\eqref{eq:sdp} are applied to solve three types of instances. Interestingly, both methods are able to return rank-one solutions for almost all the test instances, and thus guarantee the optimality of the original problem~\eqref{eq:modeltensor}. In case a rank-one solution fails to obtain, one can slightly perturb the original tensor to lead a success (see Example~\ref{ex:pertube}). All the numerical experiments are conducted using an Intel Core i5-4200M 2.5GHz computer with 4GB of RAM. The supporting software is MATLAB R2015a. To solve the convex optimization problems, CVX 2.1~\cite{GB17} and the ADMM approach in~\cite{JMZ15} are called.

\subsection{Quartic minimization from radar wave form design}

In radar system, one always regulates the interference power produced by unwanted returns through controlling the range-Doppler response~\cite{ADJZ13}. It is important to design a suitable radar waveform minimizing the disturbance power at the output of the matched filter. This can be written as
$$
\phi(\bs)=\sum_{r=0}^{n-1} \sum_{j=1}^{m} \rho(r,k) \left|\bs^{\HH}J^r (\bs\odot\bp(x_j))\right|^2,
$$
where $J^r\in\R^{n^2}$ is the shifted matrix for $r\in\{0,1,\dots,n-1\}$, $\odot$ denotes the Hadamard product, $\bp(v)=(1,e^{\ii 2\pi v},\dots, e^{\ii 2(n-1)\pi v})^{\T}$, and $\rho(r,k) = \sum_{k=1}^{n_0}{\delta_{r,r_k}}{1}_{\Delta_k}(j)\frac{\sigma_k^2}{|\Delta_k|}$
with $\delta_{r,r_k}$ being the Kronecker delta and ${1}_{\Delta_k}(j)$ being an indicator function for the index set $\Delta_k$ of discrete frequencies. Interested readers are referred to~\cite{ADJZ13} for more details of the ambiguity function and radar waveform design.

To account for the finite energy transmitted by the radar it is assumed that $\|\bs\|^2 = 1$ and a similarity constraint,
$\| \bs - \bs_0 \|^2 \le \gamma$,
needs to be enforced to obtain phase-only modulated waveforms, where $\bs_0$ is a known code sharing some nice properties. Noticing that $\|\bs\|=1$ and $\bs_0$ is known, this similarity constraint can be realized by penalizing the quantity $-|\bs^{\HH}\bs_0|$ in the objective function $\phi(\bs)$. Therefore, the following quartic minimization problem is arrived (see~\cite{JLZ16} for a detail discussion on the modelling):
\begin{equation}\label{eq:radar}
  \min_{\|\bs\| = 1} \left(\phi(\bs) - \rho |\bs^{\HH}\bs_0|^2 \| \bs\|^2\right)
\end{equation}
with a penalty parameter $\rho>0$. The objective function of~\eqref{eq:radar} is a real-valued quartic conjugate form, i.e., there is a fourth-order CPS tensor $\TT$ such that $\TT(\ov\bs^2\bs^2)=\phi(\bs) - \rho |\bs^{\HH}\bs_0|^2 \| \bs\|^2$. This shows that~\eqref{eq:radar} is an instance of~\eqref{eq:modeltensor}, which is equivalent to~\eqref{eq:modelmatrix}.


We use the data considered in~\cite{ADJZ13} to construct $\phi(\bs)$ and let $\rho=30$ in~\eqref{eq:radar}. To obtain phase-only modulated waveforms, a known code $\bs_0$ (see e.g.,~\cite{HSL10}) with $|(s_0)_i|=1$ for $i=1,\dots,n$ is chosen and further normalized such that $\|\bs_0\|=1$. The problem is solved by the nuclear norm penalty model~\eqref{eq:nuclear} and the SDP relaxation method~\eqref{eq:sdp}, respectively. In the experiment, we randomly generate $\bs_0$ for 100 instances and record the number of instances that the corresponding method outputs rank-one solutions in Table~\ref{tab2}. The convex relaxation models are solved by the ADMM algorithm in~\cite{JMZ15}, whose average CPU time (in seconds) is also reported. Observed in Table~\ref{tab2}, both convex relaxation methods always obtain rank-one solutions, leading to optimal solutions of~\eqref{eq:radar}. In terms of the speed, nuclear norm penalty method runs generally faster than SDP relaxations.

\begin{table}[h]
\caption{Efficiency for the radar wave form design}
\label{tab2}
\centering
\begin{tabular}{|c|c|c|c|c|}
\hline
$n$ & \multicolumn{2}{|c|}{Nuclear norm penalty~\eqref{eq:nuclear}}&  \multicolumn{2}{|c|}{SDP relaxation~\eqref{eq:sdp}} \\ \hline
   & rank-one & CPU & rank-one & CPU \\ \hline
 5 & 100 \% & 0.371  & 100 \% & 0.693  \\
10 & 100 \% & 4.552  & 100 \% &11.261 \\
\hline
\end{tabular}
\end{table}

\subsection{Randomly generated CPS tensors}

The data from~\eqref{eq:radar} has its own structure. In this part, we test the two relaxation methods extensively using randomly generated CPS tensors. The aim is to check the chance of getting rank-one solutions and hence generating optimal solutions for the largest eigenvalue problem~\eqref{eq:modeltensor}, under the tractability of solving the two convex relaxation models.
These CPS tensors are generated as follows. First, we randomly generate two real tensors $\U,\V\in\R^{n^4}$ whose entries follow i.i.d.\ standard normal distributions, independently. We then let $\W=\U+\ii\V$ to define a complex tensor in $\C^{n^4}$. To make it being PS, we further let $\X\in\PS^{n^4}$ where
$$
\X_{ijk\ell}=\frac{1}{4}\left(\W_{ijk\ell}+\W_{jik\ell}+\W_{ij\ell k}+\W_{ji\ell k}\right) \quad\forall\,1\le i,j,k,\ell\le n.
$$
Finally, to make it being CPS, we let $\TT=\frac{1}{2}(\X+\X^{\HH})\in\CPS^{n^4}$.

For various $n\le 15$, 100 random CPS tensor instances are generated. We then solve the two convex relaxation models~\eqref{eq:nuclear} and~\eqref{eq:sdp} using the ADMM algorithm, and record the number of instances that produce rank-one solutions. The results are shown in Table~\ref{table2} together with the average CPU time (in seconds). It shows that both the nuclear norm penalty method and the SDP relaxation model~\eqref{eq:sdp} are able to generate rank-one solutions for most randomly generated instances, and thus find the largest eigenvalue of CPS tensors. Opposite to the data of radar wave form design, SDP relaxation outperforms nuclear norm penalty method, both in speed and in the chance of optimality when the dimension of the problem increases.

\begin{table}[h]
	\caption{Efficiency for largest eigenvalue of random CPS tensors}
	\label{table2}
	\centering
\begin{tabular}{|c|c|c|c|c|}
\hline
$n$ & \multicolumn{2}{|c|}{Nuclear norm penalty~\eqref{eq:nuclear}}&  \multicolumn{2}{|c|}{SDP relaxation~\eqref{eq:sdp}} \\ \hline
   & rank-one & CPU & rank-one & CPU \\ \hline
		4  & 100 \% &0.231  & 100 \% & 0.083 \\
		6  & 100 \% &1.434  & 100 \% & 0.473 \\
		8  & 100 \% &5.570  & 100 \% & 1.824 \\
		9  & 100 \% &11.414 & 100 \% & 3.519 \\
		10 & 100 \% &19.339 & 100 \% & 5.740 \\
		12 & 99  \% &56.299 & 99 \%& 16.050 \\
		15 & 45  \% &206.879 & 82 \%& 90.184 \\
		\hline
	\end{tabular}
\end{table}

\subsection{Computing largest US-eigenvalues}

Motivated by the geometric measure of quantum entanglement, Ni et al.~\cite{NQB14} introduced the notion of unitary symmetric eigenvalue (US-eigenvalue) and unitary symmetric eigenvector (US-eigenvector). The geometric measure of entanglement has various applications such as entanglement witnesses and quantum computation. The US-eigenvalues and US-eigenvectors reflect some specific states of the composite quantum system to certain extent. Specifically, $\lambda\in\C$ is a US-eigenvalue associated with a US-eigenvector $\bx\in\C^n$ of a symmetric tensor $\Z\in\C^{n^d}$ if
$$
\ov\Z(\bullet\,\bx^{d-1})=\lambda\,\ov\bx,\,\Z(\bullet,\ov\bx^{d-1})=\lambda\,\bx,\mbox{ and }\|\bx\|=1.
$$
It is known that US-eigenvalues must be real. Jiang et.\ al.~\cite{JLZ16} showed that $\lambda\in \R$ is a US-eigenvalue of a symmetric tensor $\Z\in \C^{n^d}$ if and only if $\lambda^2$ is a C-eigenvalue of the CPS tensor $\Z\otimes\ov\Z\in\C^{n^{2d}}$ and their eigenvectors are closely related. Therefore, we may resort the model~\eqref{eq:modeltensor} to find the largest US-eigenvalue with its corresponding eigenvectors for a symmetric tensor $\Z$, i.e., to solve
\begin{equation} \label{eq:useigen}
  \max_{\|\bx\|=1}(\Z\otimes\ov\Z)(\ov\bx^{d}\bx^d).
\end{equation}

In this tests, we look into the two examples in~\cite{NQB14}. We first transfer the largest US-eigenvalue problem to~\eqref{eq:useigen}, and then use the SDP relaxation model~\eqref{eq:sdp}. The hope is to find rank-one solutions and hence to obtain the largest US-eigenvalue with its corresponding eigenvectors.

\begin{example}{\em (\cite[Table 1]{NQB14}).}
Let a symmetric tensor $\Z\in\CS^{2^3}$ have entries $\Z_{111}=2$, $\Z_{112}=\Z_{121}=\Z_{211}=1$, $\Z_{122}=\Z_{212}=\Z_{221}=-1$, $\Z_{222}=1$, and others being zeros.
\end{example}
By applying the SDP relaxation method to~\eqref{eq:useigen} and solving it using CVX, we directly generate a rank-one solution. In other words, we obtain a C-eigenpair $(\lambda^2,\bx)$ with $\lambda\in\R$ of the CPS tensor $\Z\otimes\ov\Z$, i.e.,
$$
\lambda^2 = (\Z\otimes\ov\Z)(\ov\bx^3\bx^3)= \langle \ov\Z\otimes\Z, \ov\bx^{\otimes 3}\otimes\bx^{\otimes 3} \rangle =|\ov\Z(\bx^3)|^2.
$$
However, $\ov\Z(\bx^3)$ may not be real. This can be easily done by rotating $\bx$. In particular, by letting $\bz=e^{-\ii\theta/3}\bx$ where $\theta=\arg(\ov\Z(\bx^3))$, one has $\ov\Z(\bz^3) = e^{-\ii\theta} \ov\Z(\bx^3) \in\R$. This implies that $(\lambda,\bz)$ is the corresponding US-eigenpair of $\Z$. In this example, it recovers the largest US-eigenvalue $2.3547$ with its corresponding US-eigenvector $(0.9726,0.2326)^{\T}$.

\begin{example} {\em (\cite[Table 2]{NQB14}).} \label{ex:pertube}
Let a symmetric tensor $\Z\in\CS^{2^3}$ have entries $\Z_{111}=2$, $\Z_{112}=\Z_{121}=\Z_{211}=-1$, $\Z_{122}=\Z_{212}=\Z_{221}=-2$, $\Z_{222}=1$, and others being zeros.
\end{example}
We again consider the SDP relaxation method and resort to CVX for a solution. Unfortunately, it fails to give us a rank-one solution. Motivated by the high frequency of rank-one solutions obtained when solving randomly generated tensors as shown in Table~\ref{table2}, we now add a tiny random perturbation $\E\in\CS^{2^3}$ with $\|\E\|=10^{-4}$ to the original tensor $\Z$. The hope is to generate a rank-one solution via SDP relaxation while keeping the original US-eigenpair almost unchanged since $\E$ is small enough. Furthermore, the largest US-eigenvalue may have more than one US-eigenvectors, i.e.,~\eqref{eq:useigen} admits multiple global optimal solutions. In our experiments, we observe that adding tiny perturbations not only obtains a rank-one solution, but also helps to generate different rank-one solutions under different perturbations. Using this approach, we successfully obtain the largest US-eigenvalue $3.1623$ and its four US-eigenvectors: $(0.6987+0.1088\,\ii,-0.1088+0.6987\,\ii)^{\T}$, $(0.6987-0.1088\,\ii,-0.1088-0.6987\,\ii)^{\T}$, $(-0.2551+0.6595\,\ii,0.6595+0.2551\,\ii)^{\T}$ and $(-0.2551-0.6595\,\ii,0.6595-0.2551\,\ii)^{\T}$, which are consistent with the results in~\cite{NQB14}. In fact, our convex relaxation approach is able to certify that the obtained eignevalue is globally the largest as long as the solution to~\eqref{eq:sdp} is rank-one. This certificate, however, cannot be seen from the solutions obtained in~\cite{NQB14}. Therefore, our experiment on Example~\ref{ex:pertube} helps to verify that the largest one among all the eigenvalues obtained in~\cite[Table 2]{NQB14} is actually the largest eigenvalue of $\Z$.

To conclude the numerical results, the convex relaxation methods proposed in Section~\ref{sec:method} and established based on the rank-one equivalence, are capable to find optimal solutions for the best rank-one approximation or the largest eigenvalue of CPS tensors. At least they are able to generate rank-one solutions (hence optimality) for the three types of instances discussed above. In case a rank-one solution fails to obtain, one may slightly perturb the original tensor, and the chance to obtain rank-one solutions may increase. This is one of the research topic to look into further.

We end this paper with a short concluding remark. CPS tensors, formally proposed not long ago and appeared quite often in applications, are showing growing importance. Apart from the findings on ranks, decompositions and approximations, there are also many interesting new phenomena on the numerical range and unitary decompositions, which are under our research agenda.


\end{document}